\documentclass[11pt]{article}

\usepackage{amssymb}
\makeatletter

\usepackage[english]{babel}
\usepackage{amsmath}
\usepackage{amssymb}
\usepackage{graphicx}
\usepackage{color}

\usepackage{graphicx}
\usepackage{color}
\usepackage{amsmath,amsthm,amsfonts,epsfig,setspace}
\numberwithin{equation}{section}

\newcommand{\ds}{\displaystyle}
\def\nm{\noalign{\medskip}}
\newtheorem{thm}{Theorem}[section]
\newtheorem{rmk}{Remark}[section]
\newtheorem{cor}{Corollary}[section]
\newtheorem{definition}{Definition}
\newtheorem{lem}{Lemma}[section]

\newtheorem{prop}{Propsition}[section]

\newtheorem{cond}{Condition}
 \def\p{\partial}
\def \Vh0{\stackrel{\circ}{V}_h} 
   
  \def\om{\omega}
   
\def\l{\label}  \def\f{\frac}  

   \def\eps{\varepsilon}

\def\e{\eta}

\def\l|{\left|}
\def\r|{\right|}

\newcommand{\R}{\mathbb{R}}

\newcommand\numberthis{\refstepcounter{equation}\tag{\theequation}}

\newcommand{\lc}
{\mathrel{\raise2pt\hbox{${\mathop<\limits_{\raise1pt\hbox
{\mbox{$\sim$}}}}$}}}

\newcommand{\gc}
{\mathrel{\raise2pt\hbox{${\mathop>\limits_{\raise1pt\hbox{\mbox{$\sim$}}}}$}}}

\newcommand{\ec}
{\mathrel{\raise2pt\hbox{${\mathop=\limits_{\raise1pt\hbox{\mbox{$\sim$}}}}$}}}

\def\be{\begin{equation}} \def\ee{\end{equation}}

\def\bea{\begin{eqnarray}}  \def\eea{\end{eqnarray}}

\def\beas{\begin{eqnarray*}} \def\eeas{\end{eqnarray*}}

\def\bn{\begin{enumerate}} \def\en{\end{enumerate}}

\def\bd{\begin{description}} \def\ed{\end{description}}

\title{Mathematical analysis of plasmonic nanoparticles: the scalar  
case\thanks{\footnotesize This work was supported  by the
ERC Advanced Grant Project MULTIMOD--267184.}}
\date{}

\author{
Habib Ammari\thanks{\footnotesize Department of Mathematics, 
ETH Z\"urich, 
R\"amistrasse 101, CH-8092 Z\"urich, Switzerland (habib.ammari@math.ethz.ch). }\thanks{\footnotesize Department of Mathematics and Applications,
Ecole Normale Sup\'erieure, 45 Rue d'Ulm, 75005 Paris, France
(pierre.millien@ens.fr, matias.ruiz@ens.fr, zh.hai84@gmail.com).} 
 \and Pierre Millien\footnotemark[3] \and   Matias Ruiz\footnotemark[3] \and  Hai Zhang\footnotemark[3]
}

\begin{document}
\maketitle

\begin{abstract}
Localized surface plasmons are charge density oscillations confined to metallic nanoparticles. Excitation of localized surface plasmons by an electromagnetic field at an incident wavelength where resonance occurs results in a strong light scattering and an enhancement of the local electromagnetic fields. This paper is devoted to the mathematical modeling of plasmonic nanoparticles. Its aim is threefold: (i) to mathematically define the notion of plasmonic resonance and to analyze the shift and broadening of the plasmon resonance with changes in size and shape of the nanoparticles; (ii) to
study the scattering and absorption enhancements by plasmon resonant nanoparticles and express them in terms of the polarization tensor of the nanoparticle. Optimal bounds on the enhancement factors are also derived; (iii) to show, by analyzing the imaginary part of the Green function, that one can achieve super-resolution and super-focusing using plasmonic nanoparticles. 
For simplicity,  the Helmholtz equation is used to model electromagnetic wave propagation. 
\end{abstract}

\medskip

\bigskip

\noindent {\footnotesize Mathematics Subject Classification
(MSC2000): 35R30, 35C20.}

\noindent {\footnotesize Keywords: plasmonic resonance, Neumann-Poincar\'e operator, nanoparticle, scattering and absorption enhancements,
super-resolution imaging, layer potentials.}

 \tableofcontents

\section{Introduction} \label{sec-intro}

Plasmon resonant nanoparticles have unique capabilities of enhancing the brightness  of light and confining strong electromagnetic fields \cite{SC10}. A thriving interest for optical studies of plasmon resonant nanoparticles is due to their recently proposed use as labels in molecular biology \cite{plasmon4}. New types of cancer diagnostic nanoparticles are constantly being developed. Nanoparticles are also being used in thermotherapy as nanometric heat-generators that can be activated remotely by external electromagnetic fields \cite{baffou2010mapping}.

According to the quasi-static approximation for small particles, the surface plasmon resonance peak occurs when the particle's polarizability is maximized. Plasmon resonances in nanoparticles can be treated at the quasi-static limit as an eigenvalue problem for the Neumann-Poincar\'e integral operator, which leads to direct calculation of resonance values of permittivity and optimal design of nanoparticles that resonate at specified frequencies \cite{yatin, pierre, Gri12, plasmon1,plasmon3}. At this limit, they  are size-independent. However, as the particle size increases, they are determined from scattering and absorption blow up and become size-dependent.  This was experimentally observed, for instance, in \cite{eth, novotny, tocho}.

In \cite{pierre}, we have provided  a rigorous mathematical framework for localized surface plasmon resonances. We have considered the full Maxwell equations. Using layer potential techniques, we have derived the quasi-static limits of the electromagnetic fields in the presence of  nanoparticles. We have proved that the quasi-static limits are uniformly valid with respect to the nanoparticle's bulk electron relaxation rate. We have introduced  localized plasmonic resonances as the eigenvalues of the Neumann-Poincar\'e operator associated with the nanoparticle. We have described a general model for the permittivity and permeability of nanoparticles as functions of the frequency and rigorously justified the quasi-static approximation for surface plasmon resonances.

In this paper, we first prove that, as the particle size increases and 
crosses its critical value for dipolar approximation which is 
justified in \cite{pierre},  the plasmonic resonances become size-dependent. The resonance condition is determined from absorption and scattering blow up and depends on the shape, size and electromagnetic parameters of both the nanoparticle and the surrounding material.
Then, we precisely quantify the scattering absorption enhancements in plasmonic nanoparticles. We derive new bounds on the enhancement factors given the volume and electromagnetic parameters of the nanoparticles. At the quasi-static limit, we prove that the averages over the orientation of scattering and extinction cross-sections of a randomly oriented nanoparticle 
are given in terms of the imaginary part of the polarization tensor. Moreover, we show that the polarization tensor blows up at plasmonic resonances and derive bounds for the absorption and scattering cross-sections.  We also prove the blow-up of the first-order scattering coefficients at plasmonic resonances. The concept of scattering coefficients was introduced in \cite{lim} for scalar wave propagation problems and in \cite{lim2} for the full Maxwell equations, rendering a powerful and efficient tool for the classification of the nanoparticle shapes. Using such a concept, we have explained in \cite{scatcoef} the experimental results reported in \cite{physics1}.  Finally, we consider the super-resolution phenomenon in plasmonic nanoparticles. Super-resolution is meant to cross the barrier of diffraction limits by reducing the focal spot size. This resolution limit  applies only to light that has propagated for a distance substantially larger than its 
wavelength \cite{gang1b,gang2b}.
Super-focusing is the counterpart of super-resolution. It is a concept for  waves to be confined to a length scale significantly smaller than the diffraction limit of the focused waves. The super-focusing phenomenon is being intensively investigated in the field of nanophotonics as a possible technique to focus electromagnetic radiation in a region of order of a few nanometers beyond the diffraction limit of light and thereby causing an extraordinary enhancement of the electromagnetic fields. In \cite{hai, hai2}, a rigorous mathematical theory is developed to explain the super-resolution phenomenon in microstructures
with high contrast material around the source point. Such microstructures act like arrays of subwavelength sensors. A key ingredient is the calculation of the resonances and the Green function in the microstructure.
By following the methodology developed in \cite{hai, hai2}, we show in this paper that  one can achieve super-resolution using plasmonic nanoparticles as well.

The paper is organized as follows. In  section \ref{sect1} we introduce a layer potential formulation for plasmonic resonances and derive asymptotic formulas for the plasmonic resonances and the near- and far-fields in terms of the size. 
In section \ref{sec-multi-scatter} we consider the case of multiple plasmonic nanoparticles. Section \ref{sect3} is devoted to the study of the
scattering and absorption enhancements. We also clarify the connection between the blow up of the scattering frequencies and the plasmonic resonances. The scattering coefficients are simply the Fourier coefficients of the scattering amplitude \cite{lim, lim2}.  In section \ref{sect4} we investigate the behavior of  the scattering coefficients at the plasmonic resonances. In section \ref{sect5} we prove that using plasmonic nanoparticles one can achieve super resolution imaging. Appendix \ref{append1} is devoted to the derivation of asymptotic expansions with respect to the frequency of some boundary integral operators associated with the Helmholtz equation and a single particle.  These results are generalized to the case of multiple particles in Appendix \ref{append2}. In Appendix \ref{appen2d} we provide the technical modifications 
needed in order to study the shift in the plasmon resonance in the two-dimensional case. In Appendix \ref{appendixPT} we prove useful sum rules for the polarization tensor.

\section{Layer potential formulation for plasmonic resonances} \label{sect1}

\subsection{Problem formulation and some basic results}

We consider the scattering problem of a time-harmonic wave incident on a plasmonic nanoparticle. For simplicity, we use the Helmholtz equation instead of the full Maxwell equations. The homogeneous medium is characterized by electric permittivity $\varepsilon_m$ and magnetic permeability $\mu_m$, while the particle occupying a bounded and simply connected domain $D\Subset\mathbb{R}^3$ (the two-dimensional case is treated in Appendix \ref{appen2d}) of class $\mathcal{C}^{1,\alpha}$ for some $0<\alpha<1$ is characterized by electric permittivity $\varepsilon_c$ and magnetic permeability $\mu_c$, both of which may depend on the frequency. Assume that $\Re \eps_c <0, \Im \eps_c >0, \Re \mu_c <0, \Im \mu_c >0$ and define
$$
k_m = \omega \sqrt{\eps_m \mu_m}, \quad k_c = \omega \sqrt{\eps_c \mu_c},
$$
and
$$
\eps_D= \eps_m \chi(\R^3 \backslash \bar{D}) + \eps_c \chi(\bar{D}), \quad
\mu_D= \eps_m \chi(\R^3 \backslash \bar{D})+ \eps_c \chi({D}),
$$
where $\chi$ denotes the characteristic function. Let $u^i(x)=e^{ik_m d\cdot x}$ be the incident wave. Here,  $\om$ is the frequency and $d$ is the unit incidence direction. Throughout this paper, we assume that $\eps_m$ and $\mu_m$ are real and strictly positive and that $\Re k_c <0$ and $\Im k_c >0$.

Using dimensionless quantities,   we assume that the following set of conditions  holds.
\begin{cond} \label{condition0}
We assume that the numbers $\eps_m, \mu_m, \eps_c, \mu_c$ are dimensionless and are of order one. We also assume that the particle has size of order one and $\omega$ is dimensionless and is of order $o(1)$. 
\end{cond}
It is worth emphasizing that in the original dimensional variables $\omega$ refers to the ratio between the size of the particle and the wavelength. Moreover, the operating frequency varies in a small range and hence, the material parameters $\eps_c$ and $\mu_c$ can be assumed independent of the frequency.  

The scattering problem can be modeled by the following Helmholtz equation
\begin{equation}
\label{transverse} 
 \left\{
\begin{array} {ll}
&\ds \nabla \cdot \f{1}{\mu_D} \nabla  u+ \omega^2 \eps_D u  = 0 \quad \mbox{in } \R^3 \backslash \partial D, \\
\nm
& u_{+} -u_{-}  =0   \quad \mbox{on } \partial D, \\
\nm
&  \ds \f{1}{\mu_{m}} \f{\p u}{\p \nu} \bigg|_{+} - \f{1}{\mu_{c}} \f{\p u}{\p \nu} \bigg|_{-} =0 \quad \mbox{on } \partial D,\\
\nm
&  u^s:= u- u^{i}  \,\,\,  \mbox{satisfies the Sommerfeld radiation condition}.
  \end{array}
 \right.
\end{equation}
Here, $\partial/\partial \nu$ denotes the normal derivative  and the Sommerfeld radiation condition can be expressed in dimension $d=2,3$, as follows:
$$
\bigg| \frac{\partial u}{\partial |x|}   - i k_m u \bigg| \leq C |x|^{- (d+1)/2}
$$
as $|x| \rightarrow + \infty$ for some constant $C$ independent of $x$.

The model problem (\ref{transverse}) is referred to as the transverse magnetic case. Note that all the results of this paper hold true in the transverse electric case where $\eps_D$ and $\mu_D$ are interchanged. 

Let
\beas
F_1(x)&=& -u^i(x)= -e^{ik_md\cdot x}, \\
\nm
F_2(x) &=& -  \frac{1}{\mu_m}  \f{\p u^i}{\p \nu} (x)= - \frac{i}{\mu_m}  k_m e^{ik_md\cdot x} d \cdot \nu(x)
\eeas
with $\nu(x)$ being the outward normal at $x\in \partial D$. Let $G$ be the Green function for the Helmholtz operator $\Delta + k^2$ satisfying the Sommerfeld radiation condition. In dimension three, $G$ is given by
$$
G(x,y,k)= - \frac{e^{i k |x-y|}}{4\pi |x-y|}.
$$
By using the following single-layer potential and Neumann-Poincar\'e integral operator
\[
\begin{array}{lr}
\ds \mathcal{S}_{D}^{k} [\psi](x)=  \int_{\p D} G(x, y, k) \psi(y) d\sigma(y), & \quad x \in \R^3,\\
\nm
\ds
(\mathcal{K}_{D}^{k})^* [\psi](x)  = \int_{\p D } \f{\p G(x, y, k)}{ \p \nu(x)} \psi(y) d\sigma(y) ,  & \quad x \in \p D,
\end{array}
\]
we can represent the solution $u$ in the following form
\be \label{Helm-solution}
u(x) = \left\{
\begin{array}{lr}
u^i + \mathcal{S}_{D}^{k_m} [\psi], & \quad x \in \R^3 \backslash \bar{D},\\
\mathcal{S}_{D}^{k_c} [\phi] ,  & \quad x \in {D},
\end{array}\right.
\ee
where $\psi, \phi \in  H^{-\f{1}{2}}(\p D)$ satisfy the following system of integral equations on $\partial D$ \cite{book3}:
\be \label{Helm-syst}
\left\{
\begin{array}{lr}
\mathcal{S}_D^{k_m} [\psi] - \mathcal{S}_D^{k_c} [\phi] &= F_1,  \\
\nm
 \f{1}{\mu_m}\big(\f{1}{2}Id + (\mathcal{K}_D^{k_m})^*\big)[\psi] +  \f{1}{\mu_c} \big(\f{1}{2}Id- (\mathcal{K}_D^{k_c})^*\big)[\phi] &=F_2,
\end{array} \right.
\ee
where $Id$ denotes the identity operator. 

We are interested in the scattering in the quasi-static regime, i.e., for $\omega \ll 1$.  Note that for $\omega$ small enough, $\mathcal{S}_D^{k_c}$ is invertible \cite{book3}. We have $\phi =  (\mathcal{S}_D^{k_c})^{-1} \big(\mathcal{S}_D^{k_m}[\psi] - F_1\big)$, whereas
the following equation holds for $\psi$
\be  \label{equ-single}
\mathcal{A}_{D}(\om)[\psi] =f,
\ee
where
\begin{eqnarray}
\mathcal{A}_{D}(\om) &=&   \f{1}{\mu_m}\big( \f{1}{2}Id + (\mathcal{K}_D^{k_m})^* \big) +
\f{1}{\mu_c} \big(  \f{1}{2}Id - (\mathcal{K}_D^{k_c})^*\big)  (\mathcal{S}_D^{k_c})^{-1} \mathcal{S}_D^{k_m}, \label{Aw}\\
f  &=& F_2 + \f{1}{\mu_c} \big( \f{1}{2}Id - (\mathcal{K}_D^{k_c})^*\big)  (\mathcal{S}_D^{k_c})^{-1}[F_1]. \label{deff}
\end{eqnarray}


It is clear that
\begin{equation}
\label{clear}
\mathcal{A}_{D}(0)= \mathcal{A}_{D, 0} = \f{1}{\mu_m} \big( \f{1}{2}Id + \mathcal{K}_D^* \big) + \f{1}{\mu_c} \big( \f{1}{2}Id - \mathcal{K}_D^* \big)
= \big( \f{1}{2\mu_m} +  \f{1}{2\mu_c}\big)Id - \big( \f{1}{\mu_c} -  \f{1}{\mu_m} \big) \mathcal{K}_D^*,
\end{equation}
where the notation $\mathcal{K}_D^* = (\mathcal{K}_D^0)^*$ is used for simplicity.

We are interested in finding $\mathcal{A}_{D}(\om)^{-1}$. We first recall some basic facts about the Neumann-Poincar\'e
operator $\mathcal{K}_D^*$ \cite{book3, kang1, kang3, shapiro}.

\begin{lem}
\begin{enumerate}
\item[(i)] The following Calder\'on identity holds:
$\mathcal{K}_D \mathcal{S}_{D}= \mathcal{S}_{D}\mathcal{K}_D^*$;
\item[(ii)]
The operator $\mathcal{K}_D^*$ is self-adjoint in the Hilbert space $H^{-\f{1}{2}}(\p D)$ equipped with the following
inner product
\be \label{innerproduct}
(u, v)_{\mathcal{H}^*}= - (u, \mathcal{S}_{D}[v])_{-\f{1}{2},\f{1}{2}}
\ee
with $(\cdot, \cdot)_{-\f{1}{2}, \f{1}{2}}$ being the duality pairing between $H^{-\f{1}{2}}(\p D)$ and  $H^{\f{1}{2}}(\p D)$, which is equivalent to the original one;
\item[(iii)] Let $\mathcal{H}^*(\p D)$ be the space $H^{-\f{1}{2}}(\p D)$ with the new inner product. Let $(\lambda_j,\varphi_j) $, $j=0, 1, 2, \ldots$ be the eigenvalue and normalized eigenfunction pair of $\mathcal{K}_D^*$ in $\mathcal{H}^*(\p D)$,
then $\lambda_j \in (-\f{1}{2}, \f{1}{2}]$ and $\lambda_j \rightarrow 0$ as $j \rightarrow \infty$;
\item[(iv)]
 The following trace formula holds: for any $\psi\in \mathcal{H}^*(\p D)$, 
$$
(-\f{1}{2}Id+\mathcal{K}_D^*)[\psi] = \f{\p \mathcal{S}_D[\psi]}{\p \nu}\Big\vert_-;
$$
\item[(v)]
The following representation formula holds: for any $\psi \in H^{-1/2}(\p D)$,
$$
\mathcal{K}_D^* [\psi]
= \sum_{j=0}^{\infty} \lambda_j (\psi, \varphi_j)_{\mathcal{H}^*} \otimes \varphi_j.
$$
\end{enumerate}
\label{lem-Kstar_properties}
\end{lem}

It is clear that the following result holds.
\begin{lem} 
Let $\mathcal{H}(\p D)$ be the space $H^{\f{1}{2}}(\p D)$ equipped with the following equivalent inner product
\be
(u, v)_{\mathcal{H}}= ((-\mathcal{S}_{D})^{-1}[u], v)_{-\f{1}{2},\f{1}{2}}.
\ee
Then, 
$\mathcal{S}_{D}$ is an isometry between $\mathcal{H}^*(\p D)$ and $\mathcal{H}(\p D)$.
\end{lem}

We now present other useful observations and basic results.
The following holds. 
\begin{lem} \label{lem-H*}
\begin{enumerate}
\item[(i)] We have
$(-\f{1}{2}Id+\mathcal{K}_D^*)\mathcal{S}_D^{-1}[\chi(\p D)]=0$ with $\chi(\p D)$ being the characteristic function of $\partial D$.

\item[(ii)]
Let $\lambda_0 = \f{1}{2}$, then the corresponding eigenspace has dimension one and is spanned by the function $\varphi_0 = c \mathcal{S}_D^{-1}[\chi(\p D)]$ for some constant $c$ such that $|| \varphi_0||_{\mathcal{H}^*} =1$.

\item[(iii)] Moreover,
$\mathcal{H}^* (\partial D) = \mathcal{H}^*_0 (\partial D)\oplus\{\mu\varphi_0,\;\mu\in\mathbb{C}\}$, where $\mathcal{H}^*_0 (\partial D)$ is the zero mean subspace of $\mathcal{H}^* (\partial D)$ and $\varphi_j\in\mathcal{H}^*_0 (\partial D)$ for $j \geq 1$, i.e., 
$(\varphi_j,\chi(\p D))_{-\frac{1}{2}, \frac{1}{2}}=0$ for $j \geq 1$. Here, $\{\varphi_j\}_j$ is the set of normalized eigenfunctions of $\mathcal{K}_D^*$. 

\end{enumerate}
\end{lem}
From (\ref{clear}), it is easy to see that
\be
{\mathcal{A}}_{D, 0}[\psi] =  \sum_{j=0}^{\infty} \tau_j (\psi, \varphi_j)_{\mathcal{H}^*} \varphi_j,
\ee
where
\begin{equation} \label{deftauj}
 \tau_j = \f{1}{2\mu_m} +  \f{1}{2\mu_c} - \big( \f{1}{\mu_c}-\f{1}{\mu_m}  \big) \lambda_j.
\end{equation}

We now derive the asymptotic expansion of the operator $\mathcal{A}(\om)$ as $\omega \rightarrow 0$.
Using the asymptotic expansions in terms of $k$ of the operators $\mathcal{S}_D^k$, $(\mathcal{S}_{D}^{k})^{-1}$ and $(\mathcal{K}_{D}^{k})^*$ proved in Appendix
\ref{append1}, we can obtain the following result.
\begin{lem} \label{lem-operator-a-single} As $\omega \rightarrow 0$, 
the operator ${\mathcal{A}_{D}}(\om) : \mathcal{H}^*(\p D) \rightarrow \mathcal{H}^*(\p D)$ admits the asymptotic expansion
$$
{\mathcal{A}_{D}}(\om)= {\mathcal{A}}_{D, 0} +\om^2 {\mathcal{A}}_{D, 2}+ O(\om^3),
$$
where
\begin{equation} \label{defa2}
\mathcal{A}_{D, 2} = (\eps_m - \eps_c) \mathcal{K}_{D, 2} +
\f{\eps_m\mu_m - \eps_c \mu_c}{\mu_c} (\f{1}{2}Id - \mathcal{K}_D^*)\mathcal{S}_D^{-1} \mathcal{S}_{D, 2}.
\end{equation}
\end{lem}

\begin{proof} Recall that
\begin{equation}
\label{recall2}
\mathcal{A}_{D}(\om) =   \f{1}{\mu_m}\big( \f{1}{2}Id + (\mathcal{K}_D^{k_m})^* \big) +
\f{1}{\mu_c} \big(  \f{1}{2}Id - (\mathcal{K}_D^{k_c})^*\big)  (\mathcal{S}_D^{k_c})^{-1} \mathcal{S}_D^{k_m}.
\end{equation}
By a straightforward calculation, it follows that
\beas
(\mathcal{S}_D^{k_c})^{-1} \mathcal{S}_D^{k_m} &=& Id + \om\big(\sqrt{\eps_c\mu_c}\mathcal{B}_{D,1}\mathcal{S}_D+\sqrt{\eps_m\mu_m}\mathcal{S}_D^{-1}\mathcal{S}_{D,1}\big) + \\
\nm
&&\; \om^2\big( \eps_c\mu_c\mathcal{B}_{D,2}\mathcal{S}_D+\sqrt{\eps_c\mu_c\eps_m\mu_m}\mathcal{B}_{D,1}\mathcal{S}_{D,1}+\eps_m\mu_m\mathcal{S}_D^{-1}\mathcal{S}_{D,2} \big) + O(\om^3),\\
\nm
&=& Id + \om\big(\sqrt{\eps_m\mu_m}-\sqrt{\eps_c\mu_c}\big)\mathcal{S}_D^{-1}\mathcal{S}_{D,1} + \\
\nm
&&\; \om^2\big( (\eps_m\mu_m-\eps_c\mu_c)\mathcal{S}_D^{-1}\mathcal{S}_{D,2}+\sqrt{\eps_c\mu_c}(\sqrt{\eps_c\mu_c}-\sqrt{\eps_m\mu_m})\mathcal{S}_D^{-1}\mathcal{S}_{D,1}\mathcal{S}_D^{-1}\mathcal{S}_{D,1} \big) \\
\nm && + O(\om^3),
\eeas
where $\mathcal{B}_{D,1}$ and $\mathcal{B}_{D,2}$ are defined by (\ref{defB12}). 
Using the facts that
$$
\big(\f{1}{2}Id - \mathcal{K}_D^*\big)\mathcal{S}_D^{-1}\mathcal{S}_{D,1}=0
$$
and $$\f{1}{2}Id - (\mathcal{K}_D^{k})^* = \big( \f{1}{2}Id - \mathcal{K}_D^* \big) - k^2\mathcal{K}_{D, 2} + O(k^3),$$ the lemma immediately follows.\end{proof}

We regard ${\mathcal{A}_{D}}(\om)$ as a perturbation to the operator ${\mathcal{A}}_{D, 0}$ for small $\om$.
Using standard perturbation theory \cite{berry}, we can derive the perturbed eigenvalues and their associated eigenfunctions.
For simplicity, we consider the case when $\lambda_j$ is a  simple eigenvalue of the operator
$\mathcal{K}_D^*$.

We let
\be
R_{jl}= \big( {\mathcal{A}}_{D, 2}[\varphi_j], \varphi_l \big)_{\mathcal{H^*}},
\ee
where ${\mathcal{A}}_{D, 2}$ is defined by (\ref{defa2}). 

As $\omega$ goes to zero, the perturbed eigenvalue and eigenfunction have the following form:
\bea
\tau_j (\om) &=& \tau_j + \om^2 \tau_{j, 2}+ O(\om^3), \label{tau-single} \\
\varphi_j(\om) &=& \varphi_j + \om^2 \varphi_{j, 2} + O(\om^3), \label{eigenfun-single}
\eea
where
\bea
\tau_{j, 2} &=& R_{j j}, \label{tau_j2}\\
\varphi_{j, 2}&=& \sum_{l\neq j} \f{R_{jl}}{ \big( \f{1}{\mu_m} -  \f{1}{\mu_c} \big) (\lambda_j- \lambda_l)} \varphi_l.
\eea

\subsection{First-order correction to plasmonic resonances and field behavior at the plasmonic resonances}

We first introduce different notions of plasmonic resonance as follows. 
\begin{definition} \label{def-plasmonicFreq}
\begin{itemize}
\item[(i)] We say that $\om$ is a plasmonic resonance if
$$
| \tau_j(\om)| \ll 1 \quad \mbox{ and is locally minimal for some } \, j.
$$
\item[(ii)] We say that $\om$ is a quasi-static plasmonic resonance if
$|\tau_j| \ll 1 $ and is locally minimized for some $j$. Here, $\tau_j$ is 
defined by (\ref{deftauj}). 
\item[(iii)] We say that $\om$ is a first-order corrected quasi-static plasmonic resonance if $|\tau_j + \omega^2 \tau_{j,2}| \ll 1$ and is locally minimized for some $j$. Here, the correction term $\tau_{j,2}$ is 
defined by (\ref{tau_j2}). 
\end{itemize}
\end{definition}

Note that quasi-static resonances are size independent and is therefore a zero-order approximation of the plasmonic resonance in terms of the particle size while the first-order corrected quasi-static plasmonic resonance depends on the size of the nanoparticle (or equivalently on $\omega$ in view of the non-dimensionalization adopted herein).

We are interested in solving the equation ${\mathcal{A}_{D}}(\om)[\phi] = f$ when $\om$ is close to the resonance frequencies, i.e., when $\tau_j(\om)$ is very small for some $j$'s. In this case, the major part of the solution would be the contributions of the excited resonance modes $\varphi_j(\om)$.
We introduce the following definition.
\begin{definition} \label{def-j}
We call $J \subset \mathbb{N}$ index set of resonance if $\tau_j$'s are close to zero when $j \in J$ and are bounded from below when $j \in J^c$. More precisely, we choose a threshold number $\e_0 >0$ independent of $\om$ such that
$$
 | \tau_j | \geq \e_0 >0 \quad \mbox{for }\, j \in J^c.
$$
\end{definition}

\begin{rmk}
Note that for $j=0$, we have $\tau_0 = {1}/{\mu_m}$, which is of size one by our assumption. As a result,  throughout this paper, we always exclude $0$ from the index set of resonance $J$. 
\end{rmk}

From now on, we shall use $J$ as our index set of resonances.
For simplicity, we assume throughout that the following conditions hold.
\begin{cond} \label{condition1}
Each eigenvalue $\lambda_j$ for $j\in J$ is a  simple eigenvalue of the operator
$\mathcal{K}_D^*$.
\end{cond}
\begin{cond} \label{condition1add}
Let 
\begin{equation} \label{deflambda}
\lambda = \f{\mu_m+\mu_c}{2(\mu_m-\mu_c)}.
\end{equation}
We assume that $\lambda \neq 0$ or equivalently, $\mu_c \neq - \mu_m$. 
\end{cond}
Condition \ref{condition1add} implies that the set $J$ is finite. 

We define the projection $P_J (\om)$ such that
\[
P_J(\om)[\varphi_j(\om)] = \left\{
\begin{array}{lr}
\varphi_j(\om), \quad & j \in J,\\
0, \quad & j \in J^c.
\end{array} \right.
\]
In fact, we have
\be
P_J(\om) = \sum_{j\in J} P_j(\om) = \sum_{j\in J} \f{1}{2 \pi i} \int_{\gamma_j} (\xi  -{\mathcal{A}_{D}}(\om))^{-1} d \xi,
\ee
where $\gamma_j$ is a Jordan curve in the complex plane enclosing only the eigenvalue $\tau_j(\om)$ among all the eigenvalues.

To obtain an explicit representation of $P_J(\om)$, we consider the adjoint operator ${\mathcal{A}_{D}}(\om)^*$. By a similar perturbation argument, we can obtain its perturbed eigenvalue and eigenfunction, which have the following form
\bea
\widetilde{\tau}_j (\om) &=& \overline{\tau_j(\om)}, \\
\widetilde{\varphi_j}(\om) &=& \varphi_j + \om^2 \widetilde{\varphi}_{j, 2} + O(\om^2).
\eea
Using the eigenfunctions $\widetilde{\varphi}_j(\om)$, we can show that
\be
P_J(\om)[x] = \sum_{j\in J}\big(x, \widetilde{\varphi}_j(\om)\big)_{\mathcal{H^*}} {\varphi_j}(\om).
\ee
Throughout this paper, for two Banach spaces $X$ and $Y$, by  $\mathcal{L}(X,Y)$ we denote the set of bounded linear operators from $X$ into $Y$.

We are now ready to solve the equation ${\mathcal{A}_{D}}(\om)[\psi] =f$.
First, it is clear that
\be
\psi= {\mathcal{A}_{D}}(\om)^{-1}[f] = \sum_{j \in J} \f{\big( f, \widetilde{\varphi_j}(\om)\big)_{\mathcal{H^*}}}{ \tau_j(\om)}
+ {\mathcal{A}_{D}}(\om)^{-1} [P_{J^c}(\om) [f]].
\ee
The following lemma holds.
\begin{lem} \label{lem-residu} The norm 
 $\| {\mathcal{A}_{D}}(\om)^{-1}  P_{J^c} (\om)\|_{
 \mathcal{L}(\mathcal{H^*}(\partial D), \mathcal{H^*}(\partial D) )}$ is uniformly bounded in $\om$.
\end{lem}
\begin{proof}
Consider the operator
$$
{\mathcal{A}_{D}}(\om)|_{J^c}: P_{J^c}(\om)\mathcal{H^*}(\partial D) \rightarrow  P_{J^c}(\om)\mathcal{H^*} (\partial D).
$$
For $\om$ small enough, we can show that dist$(\sigma ( {\mathcal{A}_{D}}(\om)|_{J^c}), 0) \geq \f{\e_0}{2}$, where $\sigma ({\mathcal{A}_{D}}(\om)|_{J^c} )$ is the discrete spectrum of ${\mathcal{A}_{D}}(\om)|_{J^c}$.
Then, it follows that
$$
\| {\mathcal{A}_{D}}(\om)^{-1}  ( P_{J^c}(\om) f) \| = \| \big({\mathcal{A}_{D}}(\om)|_{P_{J^c}}  \big)^{-1}  ( P_{J^c}(\om) f) \| \lesssim \f{1}{\e_0} \exp(\f{C_1}{\e_0^2}) \| P_{J^c}(\om) f\|,
$$
where the notation $A \lesssim B$ means that $A \leq C B$ for some constant $C$.

On the other hand,
\beas
P_J(\om) f &=&  \sum_{j\in J}\big(f, \widetilde{\varphi_j}(\om)\big)_{\mathcal{H^*}} {\varphi_j}(\om)
= \sum_{j\in J}\big(f, {\varphi_j}+ O(\om) \big)_{\mathcal{H^*}} \big( {\varphi_j} + O(\om)\big) \\
&=&  \sum_{j\in J}\big(f, {\varphi_j}\big)_{\mathcal{H^*}} {\varphi_j}(\om) + O(\om).
\eeas
Thus,
$$
\| P_{J^c}(\om)\| = \| (Id - P_J(\om))\| \lesssim (1+ O(\om)),
$$
from which the desired result follows immediately.
\end{proof}

Second, we have the following asymptotic expansion of $f$ given by (\ref{deff}) with respect to $\omega$.

\begin{lem} \label{lem-f-1}
Let $$
f_1= -i\sqrt{\eps_m\mu_m}e^{ik_md\cdot z}\left(\f{1}{\mu_m}[d\cdot\nu(x)]+\f{1}{\mu_c}\big(\dfrac{1}{2}Id - \mathcal{K}_{D}^*\big) \mathcal{S}_{D}^{-1}[d\cdot (x-z)]\right)
$$
and let $z$ be the center of the domain $D$.
In the space $\mathcal{H}^* (\partial D)$, as $\omega$ goes to zero, we have
$$
f = \om f_1 + O(\om^2),
$$
in the sense that, for $\omega$ small enough,
$$
\| f - \om f_1 \|_{\mathcal{H}^*} \leq C \omega^2$$
for some constant $C$ independent of $\omega$. 
\end{lem}

\begin{proof}
A direct calculation yields
\beas
f  &=& F_2 + \f{1}{\mu_c} \big( \f{1}{2}Id - (\mathcal{K}_D^{k_c})^*\big)  (\mathcal{S}_D^{k_c})^{-1}[F_1]\\
&=& -\om\f{i}{\mu_m}\sqrt{\eps_m\mu_m} e^{ik_md\cdot z}[d\cdot\nu(x)]+O(\om^2)+\\
&& \f{1}{\mu_c} \Big(\big( \f{1}{2}Id - \mathcal{K}_D^*\big) \big( (\mathcal{S}_D)^{-1}+\om\mathcal{B}_{D,1}\big)+O(\om^2)\Big)[-e^{ik_md\cdot z}\big(\chi(\p D) +i\om\sqrt{\eps_m\mu_m}[d\cdot(x-z)]\big) + O(\om^2)]\\
&=& -\f{e^{ik_md\cdot z}}{\mu_c}\big(\dfrac{1}{2}Id - \mathcal{K}_{D}^*\big) \mathcal{S}_{D}^{-1}[\chi(\p D)] -\f{\om e^{ik_md\cdot z}}{\mu_c} \big( \f{1}{2}Id - \mathcal{K}_D^*\big)\mathcal{B}_{D,1}[\chi(\p D)]- \\
&& \quad \om i\sqrt{\eps_m\mu_m}e^{ik_md\cdot z} \left(\f{1}{\mu_m}[d\cdot\nu(x)]+\f{1}{\mu_c}\big(\dfrac{1}{2}Id - \mathcal{K}_{D}^*\big) \mathcal{S}_{D}^{-1}[d\cdot (x-z)]\right) + O(\om^2)\\
&=& -\om i\sqrt{\eps_m\mu_m} e^{ik_md\cdot z}\left(\f{1}{\mu_m}[d\cdot\nu(x)]+\f{1}{\mu_c}\big(\dfrac{1}{2}Id - \mathcal{K}_{D}^*\big) \mathcal{S}_{D}^{-1}[d\cdot (x-z)]\right) \\
&&+ O(\om^2),
\eeas
where we have made use of the facts that $$\big(\dfrac{1}{2}Id - \mathcal{K}_{D}^*\big) \mathcal{S}_{D}^{-1}[\chi(\p D)]=0$$ and  $$\mathcal{B}_{D,1}[\chi(\p D)] = c \mathcal{S}_{D}^{-1}[\chi(\p D)]$$ for some constant $c$; see again Appendix \ref{append1}.
\end{proof}


Finally, we are ready to state our main result in this section.

\begin{thm} \label{thm1}
Under Conditions \ref{condition0}, \ref{condition1}, and \ref{condition1add} the scattered field $u^s=u-u^i$ due to a single plasmonic particle has the following  representation in the quasi-static regime:
$$
u^s = \mathcal{S}_D^{k_m} [\psi],
$$
where
\beas
\psi &=& \sum_{j \in J}
\f{ \omega \big( f_1, \widetilde{\varphi}_j(\om)\big)_{\mathcal{H}^*} \varphi_j(\om) }{ \tau_j(\om)} +O(\om) ,\\
&=& \sum_{j \in J}\f{i k_m e^{ik_md\cdot z}\big(d\cdot\nu(x) ,\varphi_j\big)_{\mathcal{H}^*} \varphi_j +O(\om^2)}{ \lambda - \lambda_j +O(\om^2)}
+  O(\om)
\eeas
with $\lambda$ being given by (\ref{deflambda}). 
\end{thm}
\begin{proof} We have
\beas
\psi &=& \sum_{j \in J} \f{\big( f, \widetilde{\varphi}_j(\om)\big)_{\mathcal{H}^*} \varphi_j(\om) }{ \tau_j(\om)}
+  {\mathcal{A}_{D}}(\om)^{-1}  ( P_{J^c}(\om) f) ,\\
 &=& \sum_{j \in J}\f{\om\big( f_1, \varphi_j\big)_{\mathcal{H}^*} \varphi_j + O(\om^2)}{ \f{1}{2\mu_m} +  \f{1}{2\mu_c} - \big( \f{1}{\mu_c}-\f{1}{\mu_m}  \big) \lambda_j + O(\om^2)}
+  O(\om).
\eeas

We now compute $\big( f_1, \varphi_j\big)_{\mathcal{H}^*}$ with $f_1$ given in Lemma \ref{lem-f-1}.
We only need to show that
\be  \label{identity-1}
\left( \big(\dfrac{1}{2}Id - \mathcal{K}_{D}^*\big) \mathcal{S}_{D}^{-1}[d\cdot (x-z)]\big), \varphi_j \right)_{\mathcal{H}^*} = (d\cdot \nu(x), \varphi_j)_{\mathcal{H}^*}.
\ee
Indeed, we have
\beas
\Big((\f{1}{2}Id - \mathcal{K}_{D}^*) \mathcal{S}_{D}^{-1} [d\cdot (x-z)], \varphi_j \Big)_{\mathcal{H}^*} &=& -\Big( \mathcal{S}_{D}^{-1}[d\cdot (x-z)],   \big(\f{1}{2}Id - \mathcal{K}_{D}\big) \mathcal{S}_{D}[\varphi_j]  \Big)_{-\f{1}{2}, \f{1}{2}}\\
& =&  -\Big( \mathcal{S}_{D}^{-1}[d\cdot (x-z)],   \mathcal{S}_{D}\big(\f{1}{2}Id - \mathcal{K}_{D}^*\big)[\varphi_j]  \Big)_{-\f{1}{2}, \f{1}{2}}\\
& =&  -\Big( d\cdot (x-z),  \big(\f{1}{2}Id - \mathcal{K}_{D}^*\big)[\varphi_j]  \Big)_{-\f{1}{2}, \f{1}{2}}\\
& =&  -\Big( d\cdot (x-z), -\f{\p \mathcal{S}_{D}[\varphi_j]}{\p \nu}\Big\vert_-   \Big)_{-\f{1}{2}, \f{1}{2}}\\
& =&  \int_{\p D}\f{\p [d\cdot (x-z)]}{\p \nu}\mathcal{S}_D[\varphi_j]d\sigma \\
 && - \int_{D}\Big(\Delta [d\cdot (x-z)] \mathcal{S}_D[\varphi_j]-\Delta  \mathcal{S}_D[\varphi_j][d\cdot (x-z)]\Big)dx\\
& =&  -\Big( d\cdot \nu(x)  ,\varphi_j\Big)_{\mathcal{H}^*},
\eeas
where we have used the fact that $\mathcal{S}_{D}[\varphi_j]$ is harmonic in $D$.
This proves the desired identity and the rest of the theorem follows immediately.
\end{proof}

\begin{cor}
Assume the same conditions as in Theorem \ref{thm1}. Under the additional condition that
\begin{equation} \label{condres}
\min_{j\in J} |\tau_j(\om)| \gg \om^3,
\end{equation}
we have
$$
\psi = \sum_{j \in J} \f{i k_m e^{ik_md\cdot z}\big(d\cdot\nu(x) ,\varphi_j\big)_{\mathcal{H}^*} \varphi_j +O(\om^2)}{ \lambda - \lambda_j + \om^2 \big( \f{1}{\mu_c}-\f{1}{\mu_m}  \big)^{-1}\tau_{j, 2}}
+  O(\om).
$$
More generally, under the additional condition that
$$
\min_{j\in J} \tau_j(\om) \gg \om^{m+1},
$$
for some integer $m>2$, we have
$$
\psi = \sum_{j \in J} \f{i k_m e^{ik_md\cdot z}\big(d\cdot\nu(x) ,\varphi_j\big)_{\mathcal{H}^*} \varphi_j +O(\om^2)}{ \lambda - \lambda_j + \om^2 \big( \f{1}{\mu_c}-\f{1}{\mu_m}  \big)^{-1}\tau_{j, 2}+\cdots +\om^m \big( \f{1}{\mu_c}-\f{1}{\mu_m}  \big)^{-1}\tau_{j, m}}
+  O(\om).
$$
\end{cor}

Rescaling back to original dimensional variables, we suppose that the magnetic permeability $\mu_c$ of the nanoparticle is changing with respect to the operating angular frequency $\omega$ while that of the surrounding medium, $\mu_m$, is independent of $\omega$. Then we can write
\begin{equation}
\begin{aligned}
\mu_c(\omega)= \mu'(\omega) + i \mu''(\omega).
\end{aligned}
\end{equation}
Because of causality, the real and imaginary parts of  $\mu_c$ obey the following Kramer--Kronig relations: 
\begin{equation}
\begin{aligned}
&\mu''(\omega)=- \frac{1}{\pi} \mathrm{p.v. } \int_{-\infty}^{+\infty}\frac{1}{\omega-s}\mu'(s)ds,\\
&\mu'(\omega)= \frac{1}{\pi} \mathrm{p.v. } \int_{-\infty}^{+\infty}\frac{1}{\omega-s}\mu''(s)ds,
\end{aligned}
\end{equation} 
where $\mathrm{p.v.}$ stands for the principle value. 

The magnetic permeability $\mu_c(\omega)$  can be described by the Drude model; see, for instance, \cite{SC10}. We have
\be \label{drude}
\mu_c(\omega)=\mu_0(1-F\frac{\omega^2}{\omega^2-\omega_0^2+i\tau^{-1} \omega}),
\ee
where $\tau >0$ is the nanoparticle's bulk electron relaxation rate ($\tau^{-1}$ is the damping coefficient), $F$ is a filling factor, and $\omega_0$ is a localized plasmon resonant frequency.  When
$$ (1-F) (\omega^2 - \omega_0^2)^2 - F \omega_0^2 (\omega^2 - \omega_0^2) + \tau^{-2} \omega^2 < 0, $$ the real part of  $\mu_c(\omega)$ is negative.

 We suppose that $D= z + \delta B$. 
The quasi-static plasmonic resonance is defined by $\omega$ such that 
$$
\Re \f{\mu_m+\mu_c(\omega)}{2(\mu_m-\mu_c(\omega))} = \lambda_j
$$
for some $j$, where $\lambda_j$ is an eigenvalue of the Neumann-Poincar\'e operator $\mathcal{K}^*_D (=\mathcal{K}^*_B)$. It is clear that such definition is independent of the nanoparticle's size. In view of (\ref{tau-single}), the shifted plasmonic resonance is defined by 
$$\ds \mathrm{argmin} \bigg| \f{1}{2\mu_m} +  \f{1}{2\mu_c(\omega)} - \big( \f{1}{\mu_c(\omega)}-\f{1}{\mu_m}  \big) \lambda_j + \omega^2 \delta^2 \tau_{j,2}  \bigg|,$$
where $\tau_{j,2}$ is given by (\ref{tau_j2}) with $D$ replaced by $B$.

\section{Multiple plasmonic nanoparticles} \label{sec-multi-scatter}

\subsection{Layer potential formulation in the multi-particle case}
We consider the scattering of an incident time harmonic wave $u^i$ by multiple weakly coupled plasmonic nanoparticles in three dimensions.
For ease of exposition, we consider the case of $L$ particles with an identical shape.
We assume that the following condition holds.
\begin{cond} \label{cond-multi}
All the identical particles have size of order $\delta$ which is a small parameter and the distances between neighboring ones are of order one.
\end{cond}
We write $D_l = z_l + \delta \widetilde{D}$, $l=1, 2,\ldots, L$, where $\widetilde{D}$ has size one and is centered at the origin. Moreover, we denote $D_0= \delta \widetilde{D}$ as our reference nanoparticle.
Denote by
$$
D= \bigcup_{l=1}^L D_l, \quad
\eps_D= \eps_m \chi(\R^3 \backslash \bar{D}) + \eps_c \chi(\bar{D}), \quad
\mu_D= \mu_m \chi(\R^3 \backslash \bar{D})+ \mu_c \chi({D}).
$$

The scattering problem can be modeled by the following Helmholtz equation:
\begin{equation} \label{mulhel}
\left\{
\begin{array} {ll}
&\ds \nabla \cdot \f{1}{\mu_D} \nabla u+ \omega^2 \eps_D u  = 0 \quad \mbox{in } \R^3 \backslash \partial D, \\
\nm
& u_{+} -u_{-}  =0    \quad \mbox{on } \partial D, \\
\nm
&  \ds \f{1}{\mu_{m}} \f{\p u}{\p \nu} \bigg|_{+} - \f{1}{\mu_{c}} \f{\p u}{\p \nu} \bigg|_{-} =0 \quad \mbox{on } \partial D, \\
\nm
&  u^s:= u - u^i  \,\,\,  \mbox{satisfies the Sommerfeld radiation condition}.
  \end{array}
 \right.
\end{equation}
Let
\beas
u^i(x)&=& e^{ik_m d\cdot x},\\
F_{l,1}(x)&=& \ds -u^i(x) \big|_{\p D_l}= -e^{ik_md\cdot x} \big|_{\p D_l}, \\
\nm
F_{l,2}(x)&=& \ds -\f{\p u^i}{\p \nu}(x) \bigg|_{\p D_l}= -i k_m e^{ik_md\cdot x} d \cdot \nu(x) \big|_{\p D_l},
\eeas
and define the operator $\mathcal{K}_{D_p, D_l}^{k}$ by
$$
\mathcal{K}_{D_p, D_l}^{k} [\psi] (x) = \int_{\p D_p} \f{ \p G(x, y, k)}{\p \nu(x)} \psi(y) d\sigma(y),  \quad x \in  \p {D}_l.
$$
Analogously, we define
$$
\mathcal{S}_{D_p, D_l}^{k} [\psi] (x) = \int_{\p D_p} { G(x, y, k)}\psi(y) d\sigma(y),  \quad x \in  \p {D}_l.
$$
The solution $u$ of (\ref{mulhel}) can be represented as follows:
\[
u(x) = \left\{
\begin{array}{ll}
\ds
u^i + \sum_{l=1}^L\mathcal{S}_{D_l}^{k_m} [\psi_l], & \quad x \in \R^3 \backslash \bar{D},\\
\nm \ds
\sum_{l=1}^L\mathcal{S}_{D_l}^{k_c} [\phi_l],  & \quad x \in {D},
\end{array}\right.
\]
where $\phi_l, \psi_l  \in H^{-\f{1}{2}}(\p D_l)$ satisfy the following system of integral equations
\[
\left\{
\begin{array}{l}
\ds
\mathcal{S}_{D_l}^{k_m} [\psi_l] - \mathcal{S}_{D_l}^{k_c} [\phi_l] +
\sum_{p\neq l} \mathcal{S}_{D_p, D_l}^{k_m} [\psi_p]  = F_{l,1},  \\
\nm
\ds
 \f{1}{\mu_m}\big(\f{1}{2}Id + (\mathcal{K}_{D_{l}}^{k_m})^*\big)[\psi_l] +  \f{1}{\mu_c} \big(\f{1}{2}Id- (\mathcal{K}_{D_l}^{k_c})^*\big)[\phi_l] \\
\nm
\qquad \ds  + \f{1}{\mu_m}
 \sum_{p \neq l}   \mathcal{K}_{D_p, D_l}^{k_m} [\psi_p]
  =F_{l,2},
\
\end{array} \right.
\]
and
$$
\left\{
\begin{array}{ll}
\ds F_{l,1} = - u^i \quad \mbox{on } \partial D_l, \\
\nm
\ds
F_{l,2} = - \frac{1}{\mu_m}  \f{\p u^i}{\p \nu}  \quad \mbox{on } \partial D_l.
\end{array} \right.
$$

\subsection{First-order correction to plasmonic resonances and field behavior at plasmonic resonances in the multi-particle case}

We consider the scattering in the quasi-static regime, i.e., when the incident wavelength is much greater than one.
With proper dimensionless analysis, we can assume that  $\omega \ll 1$. As  a consequence,  $\mathcal{S}_D^{k_c}$ is invertible.
Note that
$$
\phi_l =  (\mathcal{S}_{D_l}^{k_c})^{-1} \big(\mathcal{S}_{D_l}^{k_m} [\psi_l] + \sum_{p\neq l}\mathcal{S}_{D_p, D_l}^{k_m} [\psi_p]- F_{l,1} \big).
$$
We obtain
the following equation for $\psi_l$'s,
$$
\mathcal{A}_{D}(w)[\psi] =f,
$$
where
\[
\mathcal{A}_{D}(w) = \begin{pmatrix}
\mathcal{A}_{D_1} (\om) &  &  & \\
& \mathcal{A}_{D_2}(\om) & & \\
& & \ddots & \\
& & & \mathcal{A}_{D_L} (\om)
\end{pmatrix}
+ \begin{pmatrix}
0 & \mathcal{A}_{1, 2} (\om) &  \cdots & \mathcal{A}_{1, L} (\om)\\
 \mathcal{A}_{2, 1}(\om) & 0 & \cdots &\mathcal{A}_{2, L} (\om) \\
 \vdots &  \cdots &  0 &\vdots\\
 \mathcal{A}_{L, 1} (\om) &  \cdots &\mathcal{A}_{L, L-1}(\om) & 0 \\
 \end{pmatrix},
 \]
\[
\psi=\begin{pmatrix}
\psi_1 \\
\psi_2\\
\vdots \\
\psi_L
\end{pmatrix}, \,\, f=\begin{pmatrix}
f_1 \\
f_2 \\
\vdots\\
f_L
\end{pmatrix},
\]
and
\beas
\mathcal{A}_{l, p} (\om) &=&
\f{1}{\mu_c}\big(\f{1}{2} Id - (\mathcal{K}_{D_l}^{k_c})^*\big) (\mathcal{S}_{D_l}^{k_c})^{-1}\mathcal{S}_{D_l, D_p}^{k_m}
+ \f{1}{\mu_m}\mathcal{K}_{D_l, D_p}^{k_m},\\
f_l  &=& F_{l,2} + \f{1}{\mu_c} \big( \f{1}{2}Id - (\mathcal{K}_{D_l}^{k_c})^*\big)  (\mathcal{S}_{D_l}^{k_c})^{-1}[F_{l,1}].
\eeas

The following asymptotic expansions hold.
\begin{lem}  \label{lem-esti-multi-1}
\begin{enumerate}
\item[(i)]
Regarded as operators from $\mathcal{H}^*(\p D_p)$ into $\mathcal{H}^*(\p D_l)$, we have
$$
\mathcal{A}_{D_j} (\om) = \mathcal{A}_{D_j, 0} + O(\delta^2 \om^2),
$$
\item[(ii)]
Regarded as operators from $\mathcal{H}^*(\p D_l)$ into $\mathcal{H}^*(\p D_j)$, we have
$$
\mathcal{A}_{l, p} (\om) = \f{1}{\mu_c}\big(\f{1}{2} Id - \mathcal{K}_{D_l}^* \big) \mathcal{S}_{D_l}^{-1}
\big( \mathcal{S}_{l, p,0,1} + \mathcal{S}_{l, p,0,2}
\big)
+ \f{1}{\mu_m}\mathcal{K}_{l, p, 0, 0}  +O(\delta^2 \omega^2)  + O(\delta^4).
$$
Moreover, 
\beas
\big(\f{1}{2} Id - \mathcal{K}_{D_l}^* \big)\circ \mathcal{S}_{D_l}^{-1} \circ \mathcal{S}_{l, p,0,1} &=& O(\delta^2),\\
\big(\f{1}{2} Id - \mathcal{K}_{D_l}^* \big)\circ \mathcal{S}_{D_l}^{-1} \circ \mathcal{S}_{l, p,0,2} &=& O(\delta^3),\\
\mathcal{K}_{l, p, 0, 0} &=& O(\delta^2).
\eeas
\end{enumerate}
\end{lem}

\begin{proof}
The proof of (i) follows from Lemmas \ref{lem-operator-a-single} and  \ref{lem-appendix21}.
We now prove (ii).
Recall that
\beas
\f{1}{2} Id - (\mathcal{K}_{D_l}^{k_c})^* &=& \f{1}{2} Id - \mathcal{K}_{D_l}^* + O(\delta^2 \omega^2), \\
\nm
(\mathcal{S}_{D_l}^{k_c})^{-1} &=& \mathcal{S}_{D_l}^{-1}
- k_c \mathcal{S}_{D_l}^{-1}\mathcal{S}_{D_l, 1} \mathcal{S}_{D_l}^{-1} +O(\delta^2 \omega^2), \\
\nm
\mathcal{S}_{D_l, D_p}^{k_m}
 &=&
\mathcal{S}_{l, p,0,0} + \mathcal{S}_{l, p,0,1} + \mathcal{S}_{l, p,0,2}
+ k_m \mathcal{S}_{l, p, 1} + k^2_m \mathcal{S}_{l, p,2,0} + O(\delta^4)+ O(\omega^2\delta^2) \\
\nm
\mathcal{K}_{D_p, D_l}^{k_m}
&=&  \mathcal{K}_{l, p, 0, 0} + O(\omega^2\delta^2).
\eeas
Using the  identity
$$
\big(\f{1}{2}Id - \mathcal{K}_{D_l}^*\big) \mathcal{S}_{D_l}^{-1}[\chi(D_l)]=0,
$$
we can derive that
\beas
\mathcal{A}_{l, p} (\om) &=&
\f{1}{\mu_c}\big(\f{1}{2} Id - \mathcal{K}_{D_l}^* \big) (\mathcal{S}_{D_l}^{k_c})^{-1}\mathcal{S}_{D_l, D_p}^{k_m}
+ \f{1}{\mu_m}\mathcal{K}_{l, p, 0, 0} +O(\delta^2 \omega^2) \\
&& =  \f{1}{\mu_c}\big(\f{1}{2} Id - \mathcal{K}_{D_l}^* \big) \mathcal{S}_{D_l}^{-1}
\mathcal{S}_{D_l, D_p}^{k_m}
+ \f{1}{\mu_m}\mathcal{K}_{l, p, 0, 0}  +O(\delta^2 \omega^2) \\
&& =  \f{1}{\mu_c}\big(\f{1}{2} Id - \mathcal{K}_{D_l}^* \big) \mathcal{S}_{D_l}^{-1}
\big( \mathcal{S}_{l, p,0,0} + \mathcal{S}_{l, p,0,1} + \mathcal{S}_{l, p,0,2}
+ k_m \mathcal{S}_{l, p, 1} + k^2_m \mathcal{S}_{l, p, 2, 0} + O(\delta^4) \big)  \\
&& + \f{1}{\mu_m}\mathcal{K}_{l, p, 0, 0}  +O(\delta^2 \omega^2) \\
&& =  \f{1}{\mu_c}\big(\f{1}{2} Id - \mathcal{K}_{D_l}^* \big) \mathcal{S}_{D_l}^{-1}
\big( \mathcal{S}_{l, p,0,1} + \mathcal{S}_{l, p,0,2}
\big)
+ \f{1}{\mu_m}\mathcal{K}_{l, p, 0, 0}  +O(\delta^2 \omega^2)  + O(\delta^4).
\eeas
The rest of the lemma follows from Lemmas \ref{lem-appendix21} and \ref{lem-appendix22}.
\end{proof}

Denote by $\mathcal{H}^*(\p D)= \mathcal{H}^*(\p D_1) \times \ldots \times \mathcal{H}^*(\p D_L)$, which is equipped with the inner product
 $$
(\psi, \phi)_{\mathcal{H}^*}=
\sum_{l=1}^L (\psi_l, \phi_l)_{\mathcal{H}^*(\p D_l)}.
$$
With the help of Lemma \ref{lem-esti-multi-1}, the following result is obvious.
\begin{lem}
Regarded as an operator from $\mathcal{H}^*(\p D)$ into $\mathcal{H}^*(\p D)$, we have
\beas
\mathcal{A}(\om)&=& \mathcal{A}_{D,0} + \mathcal{A}_{D,1} + O(\om^2\delta^2) +O(\delta^4),
\eeas
where
\[ \mathcal{A}_{D,0} = \left( \begin{array}{cccc}
\mathcal{A}_{D_1, 0} &  &  & \\
 & \mathcal{A}_{D_2, 0} &  &\\
 &  & \ldots&   \\
 &   &    &  \mathcal{A}_{D_L, 0}\end{array} \right), \,\,
\mathcal{A}_{D,1}=
\left( \begin{array}{cccc}
0 & \mathcal{A}_{D,1,12} & \mathcal{A}_{D,1, 13}  & \ldots \\
\mathcal{A}_{D,1, 21} & 0 & \mathcal{A}_{D, 1,23}   & \ldots\\
 &  & \ldots&   \\
\mathcal{A}_{D, 1,L 1} & \ldots  & \mathcal{A}_{D,1, L L-1}  & 0  \end{array} \right)
\]
with
\beas
\mathcal{A}_{D_l, 0} &=&
\big(\f{1}{2\mu_m} +  \f{1}{2\mu_c}\big) Id - ( \f{1}{\mu_c} -  \f{1}{\mu_m})\mathcal{K}_{D_l}^*,\\
\mathcal{A}_{D, 1 pq} &=&  \f{1}{\mu_c}\big(\f{1}{2} Id - \mathcal{K}_{D_p}^* \big) \mathcal{S}_{D_p}^{-1}
\big( \mathcal{S}_{p, q,0,1} + \mathcal{S}_{p, q,0,2}
\big)
+ \f{1}{\mu_m}\mathcal{K}_{p, q, 0, 0}.
\eeas

\end{lem}
It is evident that
\be
{\mathcal{A}}_{D,0}[\psi] =  \sum_{j=0}^{\infty} \sum_{l=1}^L\tau_{j}(\psi, \varphi_{j,l})_{\mathcal{H}^*} \varphi_{j,l},
\ee
where
\bea
 \tau_j &= & \f{1}{2\mu_m} +  \f{1}{2\mu_c} - \big( \f{1}{\mu_c} -  \f{1}{\mu_m} \big) \lambda_j, \\
 \varphi_{j,l} &= & \varphi_j e_l
\eea
with $e_l$ being the standard basis of $\R^L$.

We take ${\mathcal{A}}(\om)$ as a perturbation to the operator ${\mathcal{A}}_{D, 0}$ for small $\om$ and small $\delta$.
Using a standard perturbation argument, we can derive the perturbed eigenvalues and eigenfunctions.
For simplicity, we assume that the following conditions hold.
\begin{cond} \label{condition1-multi}
Each eigenvalue $\lambda_j$, $j\in J$,  of the operator
$\mathcal{K}_{D_1}^*$ is simple. Moreover, we have $\omega^2   \ll \delta$.
\end{cond}

In what follows, we only use the first order perturbation theory and derive the leading order term,
i.e., the perturbation due to the term ${\mathcal{A}}_{D, 1}$.
For each $l$, we define an $L \times L$ matrix $R_{l}$ by letting
\beas
R_{l,pq}&= & \big( {\mathcal{A}}_{D, 1}[\varphi_{l,p}], \varphi_{l,q} \big)_{\mathcal{H}^*} , \\
&= & \Big( \mathcal{A}_{D, 1}[\varphi_l e_p], \varphi_l e_q \Big)_{\mathcal{H}^*}, \\
&= & \big( \mathcal{A}_{D,1, pq}[\varphi_l],  \varphi_l 
\big)_{\mathcal{H}^*}.
\eeas

\begin{lem}
The matrix $R_l =(R_{l,pq})_{p,q=1,\ldots,L}$ has the following explicit expression:
\beas
R_{l,pp} &=& 0, \\
R_{l,pq} &=& \f{3}{4 \pi \mu_c} (\lambda_j - \f{1}{2})
\sum_{|\alpha|=|\beta|=1}  \int_{\p D_0} \int_{\p D_0} \f{(z_p-z_q)^{\alpha+ \beta}}{|z_p-z_q|^5} x^{\alpha}y^{\beta} \varphi_l(x)\varphi_l(y) d\sigma(x)d\sigma(y) \\
&& +  \big(\f{1}{4 \pi \mu_c}-\f{1}{4 \pi \mu_m}\big) (\lambda_j - \f{1}{2})
\int_{\p D_0} \int_{\p D_0} \f{x\cdot y}{|z_p- z_q|^3} \varphi_l(x)\varphi_l(y) d\sigma(x)d\sigma(y)\\
&=& O(\delta^3), \quad  p \neq q.
\eeas
\end{lem}
\begin{proof}
It is clear that
$R_{l,pp} = 0$. For $p \neq q$, we have
$$
R_{l,pq} = R_{l,pq}^I + R_{l,pq}^{II} + R_{l,pq}^{III},
$$
where
\beas
R_{l,pq}^I &= &
\f{1}{\mu_c}\Big( \big(\f{1}{2} Id - \mathcal{K}_{D_p}^* \big) \mathcal{S}_{D_p}^{-1}
\mathcal{S}_{p, q,0,1}[\varphi_l],  \varphi_l \Big)_{\mathcal{H}^*(\p D_l)}, \\
R_{l,pq}^{II} &= &
\f{1}{\mu_c}\Big( \big(\f{1}{2} Id - \mathcal{K}_{D_p}^* \big) \mathcal{S}_{D_p}^{-1}
\mathcal{S}_{p, q,0,2}[\varphi_l],  \varphi_l \Big)_{\mathcal{H}^*(\p D_l)}, \\
R_{l,pq}^{III} &= &
\f{1}{\mu_m} \big( \mathcal{K}_{p, q, 0, 0} [\varphi_l],  \varphi_l \big)_{\mathcal{H}^*(\p D_l)}.
\eeas
We first consider $R_{l,pq}^I$. By
the following identity
$$
\big(\f{1}{2} Id - \mathcal{K}_{D_p}^* \big) \mathcal{S}_{D_l}[\varphi_l]
= \mathcal{S}_{D_l}\big(\f{1}{2} Id - \mathcal{K}_{D_p} \big)[\varphi_l]
= (\lambda_j - \f{1}{2})  \varphi_l,
$$
we obtain
\beas
R_{l,pq}^I &= &  -\f{1}{\mu_c} \Big( \big(\f{1}{2} Id - \mathcal{K}_{D_p}^* \big) \mathcal{S}_{D_p}^{-1}
\mathcal{S}_{p, q,0,1}[\varphi_l],  \mathcal{S}_{D_l}[\varphi_l] \Big)_{L^2(\p D_l)}, \\
& & =
\f{1}{\mu_c} (\lambda_j - \f{1}{2})
\big( \mathcal{S}_{p, q,0,1}[\varphi_l],  \mathcal{S}_{D_l}[\varphi_l] \big)_{L^2(\p D_l)}.
\eeas
Using the explicit representation of $\mathcal{S}_{p, q,0,1}$ and the fact that
$(\chi(\p D_j), \phi_l)_{L^2(\p D_j)} =0 $ for $j\neq 0$, we further conclude that
$$
R_{l,pq}^I =0.
$$
Similarly, we have
\beas
R_{l,pq}^{II} &= &
\f{1}{\mu_c} (\lambda_j - \f{1}{2})
\big( \mathcal{S}_{p, q,0,2}[\varphi_l],  \mathcal{S}_{D_l}\varphi_l \big)_{L^2(\p D_l)}, \\
& & =  \f{1}{\mu_c} (\lambda_j - \f{1}{2})
\sum_{|\alpha|=|\beta|=1} \int_{\p D_0} \int_{\p D_0}
\Big( \f{3(z_p-z_q)^{\alpha+ \beta}}{4 \pi|z_p-z_q|^5} x^{\alpha}y^{\beta} +
\f{\delta_{\alpha \beta}x^{\alpha}y^{\beta}}{4 \pi |z_p- z_q|^3}  \Big)\varphi_l(x)\varphi_l(y)  d\sigma(x)d\sigma(y)\\
& & =  \f{3}{4 \pi \mu_c} (\lambda_j - \f{1}{2})
\sum_{|\alpha|=|\beta|=1}  \int_{\p D_0} \int_{\p D_0} \f{(z_p-z_q)^{\alpha+ \beta}}{|z_p-z_q|^5} x^{\alpha}y^{\beta} \varphi_l(x)\varphi_l(y) d\sigma(x)d\sigma(y) \\
&& +  \f{1}{4 \pi \mu_c} (\lambda_j - \f{1}{2})
\sum_{|\alpha|=1} \int_{\p D_0} \int_{\p D_0} \f{1}{|z_p- z_q|^3} x^{\alpha}y^{\alpha}\varphi_l(x)\varphi_l(y) d\sigma(x)d\sigma(y).
\eeas
Finally, note that
$$
\mathcal{K}_{p, q, 0, 0} [\varphi_l] = \f{1}{ 4 \pi |z_p-z_q|^3} a\cdot \nu(x) = \f{1}{ 4\pi |z_p-z_q|^3}\sum_{m=1}^3a_m \nu_m(x),
$$
where
$a_m = \big( (y-z_q)_m, \varphi_l \big)_{L^2(\p D_q)},
$
and $a=(a_1,a_2,a_3)^T$. 

By identity (\ref{identity-1}), we have
\beas
R_{l,pq}^{III} &= &
 -\f{1}{\mu_m} \big( \mathcal{K}_{p, q, 0, 0} [\varphi_l],  \varphi_l \big)_{\mathcal{H}^*(\p D_l)} \\
& & =  -\f{1}{4 \pi |z_p-z_q|^3 \mu_m}  \big( a\cdot \nu(x),   \varphi_l \big)_{\mathcal{H}^*(\p D_l)} \\
& & =  -\f{1}{4 \pi |z_p-z_q|^3\mu_m} \left( \big(\f{1}{2} Id - \mathcal{K}_{D_p}^* \big) \mathcal{S}_{D_p}^{-1}( a\cdot (x-z_p)), \varphi_l\right)_{\mathcal{H}^*(\p D_l)} \\
& & =  -\f{1}{4 \pi |z_p-z_q|^3\mu_m} (\lambda_j - \f{1}{2}) \big(  a\cdot (x-z_p), \varphi_l \big)_{L^2(\p D_p)} \\
& & =  -\f{1}{4 \pi |z_p-z_q|^3\mu_m} (\lambda_j - \f{1}{2}) \int_{\p D_0} \int_{\p D_0} x\cdot y \varphi_l(x)\varphi_l(y) d\sigma(x)d\sigma(y).
\eeas
This completes the proof of the lemma.

\end{proof}

We now have an explicit formula for the matrix $R_l$.  It is clear that $R_l$ is symmetric, but not self-adjoint. For ease f presentation, we assume the following condition.
\begin{cond} \label{cond6}
$R_l$ has $L$-distinct eigenvalues. 
\end{cond}
We remark that Condition \ref{cond6} is not essential for our analysis. Without this condition, the perturbation argument is still applicable, but the results may be quite complicated. We refer to \cite{kato} for a complete description of the perturbation theory.   

Let $\tau_{j, l}$ and $X_{j,l}= (X_{j, l, 1},\cdots,  X_{j, l, L})^T$, $l=1, 2, \ldots, L$, be the eigenvalues and normalized eigenvectors of the matrix $R_j$. Here, $T$ denotes the transpose. We remark that each $X_{j,l}$ may be complex valued and may not be orthogonal to other eigenvectors.

Under perturbation, each $\tau_j$ is splitted into the following $L$ eigenvalues of $\mathcal{A}(\om)$,
\be  \label{tau-multi}
\tau_{j,l} (\om) = \tau_j +  \tau_{j, l} + O(\delta^4) + O(\om^2 \delta^2).
\ee
The associated perturbed eigenfunctions have the following form
\be
\varphi_{j,l}(\om)= \sum_{p=1}^L X_{j, l, p} e_p \varphi_j + O(\delta^4) + O(\om^2 \delta^2).
\ee

We are interested in solving the equation ${\mathcal{A}_{D}}(\om)[\psi] = f$ when $\om$ is close to the resonance frequencies, i.e., when $\tau_j(\om)$ are very small for some $j$'s. In this case, the major part of the solution would be based on the excited resonance modes $\varphi_{j,l}(\om)$.
For this purpose, we introduce the index set of resonance $J$ as we did in the previous section for a single particle case.

We define
\[
P_J(\om) \varphi_{j,m}(\om) = \left\{
\begin{array}{lr}
\varphi_{j,m}(\om), \quad & j \in J,\\
0, \quad & j \in J^c.
\end{array} \right.
\]
In fact,
\be
P_J(\om) = \sum_{j\in J} P_j(\om) = \sum_{j\in J} \f{1}{2 \pi i} \int_{\gamma_j} (\xi  -{\mathcal{A}_{D}}(\om))^{-1} d \xi,
\ee
where $\gamma_j$ is a Jordan curve in the complex plane enclosing only the eigenvalues $\tau_{j,l}(\om)$ for $l=1, 2, \ldots, L$ among all the eigenvalues.

To obtain an explicit representation of $P_J(\om)$, we consider the adjoint operator ${\mathcal{A}_{D}}(\om)^*$. By a similar perturbation argument, we can obtain its perturbed eigenvalue and eigenfunctions. Note that the adjoint matrix $\bar{R}_j^T= \bar{R}_j$ has eigenvalues $\overline{{\tau}_{j, l}}$ and corresponding eigenfunctions $\overline{{X}_{j, l}}$. Then
the eigenvalues and eigenfunctions of ${\mathcal{A}_{D}}(\om)^*$ have the following form
\beas
\widetilde{\tau}_{j,l} (\om) &=& \tau_j + \overline{{\tau}_{j, l}} + O(\delta^4) + O(\om^2 \delta^2), \\
\widetilde{\varphi}_{j,l}(\om) &=& \widetilde{\varphi}_{j,l} +  O(\delta^4) + O(\om^2 \delta^2),
\eeas
where
$$
\widetilde{\varphi}_{j,l}= \sum_{p=1}^L \widetilde{X}_{j, l, p} e_p \varphi_j
$$
with $\widetilde{X}_{j, l, p}$ being a multiple of $\overline{{X}_{j, l, p}}$.

We normalize $\widetilde{\varphi}_{j,l}$ in a way such that the following holds
$$ 
({\varphi}_{j,p}, \widetilde{\varphi}_{j,q} )_{\mathcal{H}^*(\p D)} = \delta_{pq},
$$
which is also equivalent to the following condition
$$
\overline{{X}_{j, p}}^T \widetilde{X}_{j, q}= \delta_{pq}.
$$
Then, we can show that the following result holds.

\begin{lem}
In the space $\mathcal{H}^*(\partial D)$, as $\omega$ goes to zero, we have
$$
f = \omega f_0 + O(\omega^2 \delta^{\f{3}{2}}),
$$
where $f_0=(f_{0,1},\ldots, f_{0, L})^T$ with
$$
f_{0,l}= -i\sqrt{\eps_m\mu_m}e^{ik_md\cdot z_l} \left( \f{1}{\mu_m} d\cdot\nu(x)  +\f{1}{\mu_c}\big(\dfrac{1}{2}Id - \mathcal{K}_{D_l}^*\big) \mathcal{S}_{D_l}^{-1}[d\cdot (x-z)]\right)   =O(\delta^{\f{3}{2}}).
$$
\end{lem}

\begin{proof}
We first show that
$$
\| u\|_{\mathcal{H}^*(\partial D_0)} = \delta^{\f{3}{2} +m } \| u\|_{\mathcal{H}^*(\partial \widetilde{D})}, \quad
\| u\|_{\mathcal{H}(\partial D_0)} = \delta^{\f{1}{2} +m } \| u\|_{\mathcal{H}(\partial \widetilde{D})}
$$
for any homogeneous function $u$ such that
$u( \delta x) =\delta^m u(x)$.
Indeed, we have $\eta (u)(x) = \delta^m u(x)$.
Since $\|\eta (u)\|_{\mathcal{H}^*(\p \widetilde{D})} = \delta^{-\f{3}{2}}\|u\|_{\mathcal{H}^*(\p D_0)}$ (see Appendix \ref{append2}),
we obtain
$$
\|u\|_{\mathcal{H}^*(\p D_0)} = \delta^{\f{3}{2}}\|\eta (u)\|_{\mathcal{H}^*(\p \widetilde{D})}
= \delta^{\f{3}{2} +m } \| u\|_{\mathcal{H}^*(\partial \widetilde{D})},
$$
which proves our first claim. The second claim follows in a similar way.
Using this result, by a similar argument as in the proof of Lemma \ref{lem-f-1} we arrive at the desired asymptotic result.
\end{proof}

Denote by $Z= (Z_1, \ldots, Z_L)$, where $Z_j= ik_m e^{ik_md\cdot z_j}$.
We are ready to present our main result in this section.

\begin{thm} \label{thm2}
Under Conditions \ref{condition0}, \ref{condition1}, \ref{condition1add}, and \ref{cond-multi}, the scattered field by $L$ plasmonic particles in the quasi-static regime has the following representation
$$
u^s = \mathcal{S}_D^{k_m} [\psi],
$$
where
\beas
\psi &=& \sum_{j \in J} \sum_{l=1}^L \f{\big( f, \widetilde{\varphi}_{j,l}(\om)\big)_{\mathcal{H}^*} \varphi_{j,l}(\om) }{ \tau_{j,l}(\om)}
+  {\mathcal{A}_{D}}(\om)^{-1}  ( P_{J^c}(\om) f) \\
&=& \sum_{j \in J} \sum_{l=1}^L \f{(d\cdot \nu(x), \varphi_j)_{\mathcal{H}^*(\p D_0)} Z  \overline{{\widetilde{X}}_{j,l}} \, \varphi_{j,l}
+ O(\om^2 \delta^{\f{3}{2}})}
{ \lambda - \lambda_j + \big( \f{1}{\mu_c}-\f{1}{\mu_m}  \big)^{-1}\tau_{j, l} + O(\delta^4)+ O(\delta^2 \omega^2) }
+  O(\om \delta^{\f{3}{2}}).
\eeas


\end{thm}
\begin{proof}
The proof is similar to that of Theorem \ref{thm1}.
\end{proof}

As a consequence, the following result holds. 
\begin{cor}
With the same notation as in Theorem \ref{thm2} and under the additional condition that
$$
\min_{j\in J} |\tau_{j,l}(\om)| \gg \om^q \delta^p,
$$
for some integer $p$ and $q$, and
$$
\tau_{j,l}(\om) = \tau_{j, l, p, q} + o(\om^q \delta^p),
$$
we have
$$
\psi = \sum_{j \in J} \sum_{l=1}^L \f{(d\cdot \nu(x), \varphi_j)_{\mathcal{H}^*(\p D_0)} Z \overline{{\widetilde{X}}_{j,l}}\, \varphi_{j,l}
+ O(\om^2 \delta^{\f{3}{2}})}{\tau_{j, l, p, q}}
+   O(\om  \delta^{\f{3}{2}}).
$$
\end{cor}

\section{Scattering and absorption enhancements} \label{sect3}

In this section we analyze the scattering and absorption enhancements. 
We prove that, at the quasi-static limit, the averages over the orientation of scattering and extinction cross-sections of a randomly oriented nanoparticle 
are given by (\ref{f1import}) and (\ref{f2import}), where $M$ given by (\ref{defm}) is the polarization tensor associated with the nanoparticle $D$ and the magnetic contrast $\mu_c(\omega)/\mu_m$. In view of (\ref{eq-PT}), the polarization tensor $M$ blows up at the plasmonic resonances, which yields scattering and absorption enhancements.  A bound on the extinction cross-section is derived in (\ref{bound1}). As shown in (\ref{BoundUniv}) and (\ref{bound3}), it can be sharpened for nanoparticles of elliptical or ellipsoidal shapes. 

\subsection{Far-field expansion}
For simplicity, we assume throughout this section that $D$ contains the origin. We first prove the following representation for the scattering amplitude. 
\begin{prop}
Let $x\in \mathbb{R}^3$ be such that $\vert x \vert \gg 1/\omega$. Then, we have
\begin{align}\label{eq:farfieldexpansion}
u^s(x) = - \frac{e^{ik_m \vert x \vert }}{4\pi\vert x \vert}  
A_\infty\left(\frac{x}{\vert x \vert}\right) + O\left(\frac{1}{\vert x \vert^2}\right)
\end{align}
with
\begin{align} \label{eq-scatteringAmplitude}
A_\infty \left(\frac{x}{\vert x \vert}\right) = \int_{\partial D}  e^{- i k_m \frac{x}{\vert x \vert} \cdot y} \psi(y) d\sigma(y)
\end{align}
being the scattering amplitude and $\psi$ being defined by (\ref{Helm-syst}). 
\end{prop}

\begin{proof}  We recall that the scattered field
$u^s$ can be represented as follows:
\begin{align*}
u^s(x)=&\mathcal{S}_D^{k_m}[\psi](x)\\
 =& - \frac{1}{4\pi} \int_{\partial D} \frac{e^{ik_m \vert x-y\vert}}{\vert x-y\vert} \psi(y) d\sigma(y) .
 \end{align*}
 From
\begin{align*}
\vert x-y\vert = \vert x \vert \left(1 - \frac{x\cdot y}{\vert x \vert^2}  +O(\frac{1}{\vert x \vert^2 } ) \right),
\end{align*}
it follows that
\begin{align*}
u^s(x)= - \frac{e^{i k_m \vert x \vert }}{4\pi \vert x \vert } \int_{\partial D}  e^{-i k_m \frac{x}{\vert x \vert }\cdot y }
\psi(y)\left(1 + \frac{(x\cdot y)}{\vert x \vert^2} \right)d\sigma(y)  +o\left(\frac{1}{\vert x\vert ^2} \right),
\end{align*}
which yields the desired result. \end{proof}

\subsection{Energy flow}
The following definitions are from \cite{born1999principles}.  We include them here for the sake of completeness.
The analogous quantity of the Poynting vector in scalar wave theory is the \textit{energy flux} vector; see \cite{born1999principles}. We recall that for a real monochromatic field
\begin{align*}
U(x,t)=\mathrm{Re}\left[ u(x) e^{-i \omega t} \right],
\end{align*}
the averaged value  of the energy flux vector, taken over an interval which is long compared to the period of the oscillations, is given by
\begin{align*}
F(x)=-i C \left[ \overline{u}(x)\nabla u(x) - u(x) \nabla \overline{u}(x)  \right],
\end{align*} where $C$ is a positive constant depending on the polarization mode. In the transverse electric case, $C= {\omega/\mu_m}$ while in the transverse magnetic case $C= {\omega /\varepsilon_m}$.
We now consider the outward flow of energy through the sphere $\partial B_R$
 of radius $R$ and center the origin:
 \begin{align*}
 \mathcal{W}= \int_{\partial B_R} F(x) \cdot \nu(x) d\sigma(x),
 \end{align*}
 where $\nu(x)$ is the outward normal at $x \in \partial B_R$. 
 
As the total field can be written as $U=u^s +u^i$, the flow can be decomposed into three parts:
\begin{align*}
\mathcal{W} = \mathcal{W}^i + \mathcal{W}^s + \mathcal{W}',
\end{align*}
where
\begin{align*}
\mathcal{W}^i =&-iC\int_{\partial B_R} \left[ \overline{u^i}(x)\nabla u^i(x) - u^i(x) \nabla \overline{u^i}(x)  \right] \cdot \nu(x)\, d\sigma(x), \\
\mathcal{W}^s=&-iC\int_{\partial B_R} \left[ \overline{u^s}(x)\nabla u^s(x) - u^s(x) \nabla \overline{u^s}(x)  \right]\cdot \nu(x) \, d\sigma(x),\\
\mathcal{W}' =&-iC\int_{\partial B_R} \left[ \overline{u^i}(x)\nabla u^s(x) - u^s(x) \nabla \overline{u^i}(x)  -  u^i(x)\nabla \overline{ u^s} (x) +\overline{ u^s}(x) \nabla u^i(x)\right]\cdot \nu(x) \, d\sigma(x). 
\end{align*}
In the case where $u^i$ is a plane wave, we can see that $\mathcal{W}^i=0$:
\begin{align*}
\mathcal{W}^i=& -iC\int_{\partial B_R} \left[ \overline{u^i}(x)\nabla u^i(x) - u^i(x) \nabla \overline{u^i}(x)  \right]\, \, d\sigma(x), \\
=& -iC \int_{\partial B_R} \left[ e^{-i k_m d\cdot x } i k_m d e^{i k_m d\cdot x} +  e^{i k_m d\cdot x }  k_m d e^{-i k_m d\cdot x} \right] \cdot \nu(x)\,  d\sigma(x),\\
=& 2Ck_m d\cdot \int_{\partial B_R} \nu(x) \, d\sigma(x) , \\=&0.
\end{align*}
In a non absorbing medium with non absorbing scatterer, $\mathcal{W}$ is equal to zero because the electromagnetic energy would be conserved by the scattering process.
However, if the scatterer is an absorbing body, the conservation of energy gives the rate of absorption as \begin{align*}
\mathcal{W}^a= - \mathcal{W}.
\end{align*}
Therefore, we have
\begin{align*}
\mathcal{W}^a + \mathcal{W}^s = -\mathcal{W}'.
\end{align*}
Here, $\mathcal{W}'$ is called the extinction rate. It is the rate at which the energy  is removed by the scatterer from the illuminating plane wave, and it is the sum of the rate of absorption
and the rate  at which energy is scattered.
\subsection{Extinction, absorption, and scattering cross-sections and the optical theorem}

Denote by $U^i$ the quantity  $U^i(x)= \left\vert\overline{u^i}(x)\nabla u^i(x) - u^i(x) \nabla \overline{u^i}(x) \right\vert $.  In the case of a plane wave illumination, $U^i(x)$ is independent of $x$ and is given by $U^i=2 k_m$.
\begin{definition}
The scattering cross-section $Q^s$, the absorption cross-section $Q^a$ and the extinction cross-section are defined by
\begin{align*}
Q^s=\frac{\mathcal{W}^s}{U^i},  \quad
Q^a=\frac{\mathcal{W}^a}{U^i} ,\quad
Q^{ext}= \frac{-\mathcal{W}'}{U^i} .
\end{align*}
Note that these quantities are independent of $x$.
\end{definition}

\begin{thm}[Optical theorem]
\label{theo:optical}
If $u^i(x)=e^{ik_m d\cdot x}$, where $d$ is a unit direction, then 
\begin{align}
Q^{ext}=& Q^s+Q^a = \frac{1}{k_m}  {\Im }\left[ A_\infty(d) \right], \label{threeres}\\
Q^s=&\int_{\mathbb{S}^2} \vert A_\infty(\hat{x}) \vert^2 d\sigma(\hat{x}) \label{threeres2}
\end{align}
with $A_\infty$ being the scattering amplitude defined by (\ref{eq-scatteringAmplitude}). 
\end{thm}

%
\begin{proof} The Sommerfeld radiation condition gives, for any $x\in \partial B_R$,
\begin{align}\label{eq:sommerfeld}
\nabla u^s(x) \cdot \nu(x) \sim ik_m u^s(x).
\end{align}
Hence, from (\ref{eq:farfieldexpansion}) we get
\begin{align*}
u^s(x)\nabla\overline{u^s}(x) \cdot \nu(x) - \overline{u^s}(x)\nabla u^s(x) \cdot \nu(x)\sim \frac{2C k_m}{\vert x \vert^2} \left\vert A_\infty \left( \frac{x}{\vert x \vert }\right) \right\vert^2,
\end{align*} which yields (\ref{threeres2}).
We now compute the extinction rate.
We have
\begin{align}\label{eq:nablaplanewave}
\nabla u^i(x) \cdot \nu(x) = ik_m d\cdot \nu(x) e^{i k_m d\cdot x}.
\end{align}
Therefore, using \ref{eq:sommerfeld} and \ref{eq:nablaplanewave}, it follows that
\begin{align*}
\overline{u^i}(x)\nabla u^s(x) \cdot \nu(x) - u^s(x)\nabla \overline{u^i}(x) \cdot \nu(x) =& i k_m \frac{e^{i k_m (\vert x \vert -d\cdot x)}}{4\pi \vert x \vert} d\cdot n   - i k_m \frac{e^{ i k_m ( \vert x \vert -d\cdot x   ) } }{4\pi \vert x \vert } \\
=&\frac{- i k_m e^{i k_m \vert x \vert - d\cdot \nu(x)}}{4\pi \vert x \vert} \left(  d\cdot \nu(x) -1 \right).
\end{align*}
For $x\in \partial B_R$, we can write
\begin{align*}
\overline{u^i}(x)\nabla u^s(x) \cdot \nu(x) - u^s(x)\nabla \overline{u^i}(x) \cdot \nu(x) =\frac{ i k_m e^{-i k_m R\nu(x)\cdot\left( d-\nu(x)\right)}}{4\pi R }\left(  d\cdot \nu(x) -1 \right).
\end{align*}
We now use Jones' lemma (see, for instance, \cite{born1999principles}) to write the following asymptotic expansion as $R \rightarrow \infty$
\begin{align*}
\frac{1}{R} \int_{\partial B_R} \mathcal{G}(\nu(x)) e^{-i k_m d\cdot \nu(x)} d \sigma(x) \sim \frac{2 \pi i}{k_m} \left( \mathcal{G}(d) e^{-i k_m R} - \mathcal{G}(-d) e^{i k_m R}\right), 
\end{align*} where
\begin{align*}
\mathcal{G}(\nu(x)) = d\cdot \nu(x) -1.
\end{align*}
Hence, 
\begin{align*}
\int_{\partial B_R }\left[ \overline{u^i}(x)\nabla u^s(x)  - u^s(x)\nabla \overline{u^i}(x)\right] \cdot \nu(x)  \sim - A_\infty(d) \quad \mbox{as } R \rightarrow \infty.
\end{align*}
Therefore, 
\begin{align*}
\mathcal{W}' =& -i C\left[ A_\infty(d) - \overline{A_\infty}(d)\right] =2 C {\Im }\left[A_\infty(d)\right].
\end{align*}
Since
\begin{align*}
\vert C\vert  \left\vert\overline{u^i}(x)\nabla u^i(x) - u^i(x) \nabla \overline{u^i}(x) \right\vert = 2\vert C\vert k_m,
\end{align*} we get the result. \end{proof}

\subsection{The quasi-static limit}
We start by recalling the small volume expansion for the far-field. 
Let $\lambda$ be defined by (\ref{deflambda}) and let 
\begin{equation} \label{defm}
 M(\lambda,D ) := \int_{\partial D} (\lambda Id - \mathcal{K}_D^*)^{-1} [\nu] x \, d\sigma(x)
\end{equation}
be the polarization tensor. The following asymptotic expansion holds. It can be proved by exactly the same arguments as those in \cite{pierre}.  
\begin{prop}
Assume that $D=\delta B + z$.  As $\delta$ goes to zero the scattered field $u^s$ can be written as follows:
\begin{equation}
\begin{array}{lll} \label{eq:asymptotiquesmallvolume}
u^s(x)&=& \ds - k_m^2 \left( \frac{\varepsilon_c}{\varepsilon_m}-1 \right)\vert D \vert G(x,z,k_m) u^i(z) - \nabla_z G(x,z,k_m) \cdot M(\lambda,D ) \nabla u^i(z) 
\\ \nm &&\ds + O\left(\frac{\delta^4}{\mathrm{dist}(\lambda, \sigma(\mathcal{K}_D^*))}\right)
\end{array} \end{equation}
for $x$ away from $D$. Here, $\mathrm{dist}(\lambda, \sigma(\mathcal{K}_D^*))$ denotes $\min_j |\lambda - \lambda_j|$ with $\lambda_j$ being the eigenvalues of $\mathcal{K}_D^*$.
\end{prop}


Assume for simplicity that $\eps_c=\eps_m$. We explicitly compute the scattering amplitude $A_\infty$ in (\ref{eq:farfieldexpansion}).
Take $u^i(x)=e^{i k_m d\cdot x}$ and assume again for simplicity that $z=0$. Equation (\ref{eq:asymptotiquesmallvolume}) yields, for $\vert x \vert \gg \frac{1}{\omega}$, 
\begin{align*}
u^s(x) =  \frac{e^{ik_m \vert x \vert}}{4 \pi \vert x \vert} i k_m\left(i k_m \frac{x}{\vert x \vert} - \frac{x}{\vert x \vert^2}\right)\cdot M(\lambda, D)d  + O(\frac{\delta^4}{\mathrm{dist}(\lambda, \sigma(\mathcal{K}_D^*))}).
\end{align*}
Since we are in the far-field region, we can write that, up to an error of order ${\delta^4}/{\mathrm{dist}(\lambda, \sigma(\mathcal{K}_D^*))} $, 
\begin{align}\label{eq:farfieldamplitude}
u^s(x) = - k_m^2 \frac{e^{ik_m \vert x \vert}}{4 \pi \vert x \vert} \left( \frac{x}{\vert x \vert} \cdot M(\lambda, D)d \right) +O\left(\frac{1}{\vert x \vert^2}\right).
\end{align}
In the next proposition we write the extinction and scattering cross-sections in terms of the polarization tensor.
\begin{prop}
The leading-order term (as $\delta$ goes to zero) of the average over the orientation of the extinction cross-section of a randomly oriented nanoparticle is given by
\begin{align} \label{f1import}
Q^{ext}_m=\frac{k_m}{3}\Im \left[\mathrm{Tr } M(\lambda, D)\right],
\end{align}
where $\mathrm{Tr }$ denotes the trace of a matrix. The leading-order term of the average over the orientation scattering cross-section of a randomly oriented nanoparticle is given by
\begin{align} \label{f2import}
Q_m^s=k_m^4 \frac{16 \pi }{9} \left\vert  \mathrm{Tr } M(\lambda, D) \right\vert^2.
\end{align}
\end{prop}
\begin{proof} 
Remark from (\ref{eq:farfieldamplitude}) that the scattering amplitude $A_\infty$ in the case of a plane wave illumination is given by 
 \begin{align}\label{eq-scatteringAmplitude2}
A_\infty\left(\frac{x}{\vert x \vert}\right) =-k_m^2 \frac{x}{\vert x \vert} \cdot M(\lambda, D) d.
\end{align} Using Theorem \ref{theo:optical}, we can see that for a given orientation
\begin{align*}
Q^{ext}= -k_m {\Im }\left[ d\cdot M(\lambda, D)d \right].
\end{align*}
Therefore, if we integrate $Q^{ext}$ over all illuminations we find that
\begin{align*}
Q^{ext}_m=& \frac{k_m}{4\pi } {\Im }\left[ \int_{\mathbb{S}^2} d\cdot M(\lambda, D)d  \,  d\sigma(d) \right].
\end{align*}
Since $\Im M(\lambda, D)$ is symmetric, it can be written as $\Im M(\lambda, D) = P^t N(\lambda) P$ where $P$ is unitary and $N$ is diagonal and real. Then, by the change of variables $d=P^t x$ and using spherical coordinates, it follows that
\begin{align}
Q^{ext}_m =& \frac{k_m}{4\pi } \left[ \int_{\mathbb{S}^2} x\cdot N(\lambda)x   d\sigma(x) \right] , \nonumber\\
=& \frac{k_m}{3} \left[ \mathrm{Tr } N(\lambda) \right] = \frac{k_m}{3} {\Im }\left[ \mathrm{Tr } M(\lambda, D) \right]. \label{45f}
\end{align}
Now, we compute the averaged scattering cross-section. 
Let $\Re M(\lambda, D) = \widetilde{P}^t \widetilde{N}(\lambda) \widetilde{P}$ where $\widetilde{P}$ is unitary and $\widetilde{N}$ is diagonal and real.
We have
\begin{align*}
Q^s_m=& k_m^4\iint_{\mathbb{S}^2\times \mathbb{S}^2} \left\vert x\cdot M(\lambda, D) d \right\vert^2 \, d\sigma(x) \, d\sigma(d),\\
=& k_m^4 \bigg[ \iint_{\mathbb{S}^2\times \mathbb{S}^2} \left\vert \widetilde{x} \cdot N(\lambda) \widetilde{d} \right\vert^2 d\sigma(\widetilde{x}) d\sigma(\widetilde{d}) + \iint_{\mathbb{S}^2\times \mathbb{S}^2} \left\vert \widetilde{x} \cdot \widetilde{N}(\lambda) \widetilde{d} \right\vert^2 \, d\sigma(\widetilde{x}) \, d\sigma(\widetilde{d}) \bigg].
\end{align*}
Then a straightforward computation in spherical coordinates  gives
\begin{align*}
Q_m^s=k_m^4 \frac{16 \pi }{9} \left\vert  \mathrm{Tr } M(\lambda, D) \right\vert^2 ,
\end{align*}
which completes the proof. \end{proof}

From Theorem \ref{theo:optical}, we obtain that the averaged absorption cross-section is given by
$$
Q_m^a =   \frac{k_m}{3} {\Im }\left[ \mathrm{Tr } M(\lambda, D) \right] -
k_m^4 \frac{16 \pi }{9} \left\vert  \mathrm{Tr } M(\lambda, D) \right\vert^2.
$$
Therefore, under the condition (\ref{condres}), $Q_m^a$ blows up at plasmonic resonances.

\subsection{An upper bound for the averaged extinction cross-section}
The goal of this section is to derive an upper bound  for the modulus of the averaged extinction cross-section $Q^{ext}_m$ of a randomly oriented nanoparticle.
Recall that the entries $M_{l,m}(\lambda, D)$ of the polarization tensor $M(\lambda, D)$ are given by
\be \label{eq-polTensor}
M_{l,m}(\lambda, D) := \int_{\partial D} x_l (\lambda I - \mathcal{K}^*_{D})^{-1}[\nu_m](x)\,  d\sigma(x).
\ee
For a $\mathcal{C}^{1,\alpha}$ domain $D$ in $\mathbb{R}^d$, $\mathcal{K}_D^*$ is compact and self-adjoint in $\mathcal{H}^*$ (defined in Lemma \ref{lem-Kstar_properties} for $d=3$ and in Lemma \ref{lem-K_star_properties2d} for $d=2$). Thus, we can write
$$
(\lambda Id-\mathcal{K}_D^*)^{-1} [\psi]
= \sum_{j=0}^{\infty} \f{(\psi, \varphi_j)_{\mathcal{H}^*} \otimes \varphi_j}{\lambda-\lambda_j},
$$
with $(\lambda_j, \varphi_j)$ being the eigenvalues and eigenvectors of $\mathcal{K}_D^*$ in $\mathcal{H}^*$ (see Lemma \ref{lem-Kstar_properties}).
Hence, the entries of the polarization tensor $M$ can be decomposed as
\begin{equation} \label{eq-PT}
M_{l,m}(\lambda, D) = \sum_{j=1}^{\infty} \f{\alpha^{(j)}_{l,m}}{\lambda-\lambda_j},
\end{equation}
where $\alpha^{(j)}_{l,m} := ( \nu_m , \varphi_j )_{\mathcal{H}^*} (\varphi_j , x_l)_{-\f{1}{2},\f{1}{2}}$. Note that $(\nu_m,\chi(\p D))_{-\f{1}{2},\f{1}{2}}=0$.  So, considering the fact that $\lambda_0 = 1/2$, we have $( \nu_m , \varphi_0 )_{\mathcal{H}^*}=0$ and so, $\alpha_{l,m}^{(0)}=0$.

The following lemmas are useful for us. 
\begin{lem} \label{lem-alpha>0} We have
$$
\alpha^{(j)}_{l,l}\geq 0 ,\quad j \geq 1.
$$
\end{lem}
\begin{proof}
For $d = 3$, we have
\beas
(\varphi_j , x_l)_{-\f{1}{2},\f{1}{2}} &=& \Big(\big(\f{1}{2}-\lambda_j\big)^{-1}\big(\f{1}{2}Id - \mathcal{K}_{D}^*\big)[\varphi_j],x_l\Big)_{-\f{1}{2},\f{1}{2}}\\
&=& \f{-1}{1/2-\lambda_j}\Big(\f{\p \mathcal{S}_D[\varphi_j]}{\p \nu}\Big\vert_-,x_l\Big)_{-\f{1}{2},\f{1}{2}}\\
&=& \int_{\p D}\f{\p x_l}{\p \nu}\mathcal{S}_D[\varphi_j]d\sigma - \int_{D}\Big(\Delta x_l \mathcal{S}_D[\varphi_j]-x_l\Delta  \mathcal{S}_D[\varphi_j]\Big)dx\\
&=& \f{( \nu_l,\varphi_j)_{\mathcal{H}^*}}{1/2-\lambda_j},
\eeas
where we used the fact that $\mathcal{S}_D[\varphi_j]$ is harmonic in $D$. The same result holds for $d = 2$ if we change $\mathcal{S}_D$ by $\widetilde{\mathcal{S}}_D$ (see Appendix \ref{appen2d}). Since $|\lambda_j|<1/2$ for $j\geq1$, we obtain the result.
\end{proof}
\begin{lem} \label{lem-sumRules}
Let $$M_{l,m}(\lambda, D) = \sum_{j=1}^{\infty} \f{\alpha^{(j)}_{l,m}}{\lambda-\lambda_j}$$ be the $(l,m)$-entry of the polarization tensor $M$ associated with a $\mathcal{C}^{1,\alpha}$ domain $D\Subset \mathbb{R}^d$. Then, the following properties hold:
\begin{itemize}
\item[(i)]
\beas
\sum_{j=1}^{\infty}\alpha^{(j)}_{l,m} &=& \delta_{l,m}|D|;
\eeas
\item[(ii)] \beas
\sum_{j=1}^{\infty}\lambda_i\sum_{l=1}^d\alpha^{(j)}_{l,l} &=& \f{(d-2)}{2}|D|;\eeas
\item[(iii)]
\beas
\sum_{j=1}^{\infty}\lambda_j^2\sum_{l=1}^d\alpha^{(j)}_{l,l} &=& \f{(d-4)}{4}|D|+\sum_{l=1}^d\int_{ D} |\nabla \mathcal{S}_{D}[\nu_l]|^2dx.
\eeas
\end{itemize}
\end{lem}
\begin{proof}
The proof can be found in Appendix \ref{appendixPT}.
\end{proof}
Let $\lambda = \lambda'+i\lambda''$. We have
\begin{equation} \label{num1}
\big| \Im(\mathrm{Tr}(M(\lambda, D))) \big| = \sum_{j=1}^{\infty} \f{|\lambda''| \sum_{l=1}^d\alpha^{(j)}_{l,l}}{(\lambda'-\lambda_j)^2+\lambda''^2}.
\end{equation}

For $d=2$ the spectrum $\sigma(\mathcal{K}_D^*)\backslash\{1/2\}$ is symmetric. For $d=3$ this is no longer true. Nevertheless, for our purposes, we can assume that $\sigma(\mathcal{K}_D^*)\backslash\{1/2\}$ is symmetric by defining $\alpha^{(j)}_{l,l}=0$ if $\lambda_j$ is not in the original spectrum.

 Without loss of generality we assume for ease of notation that Conditions \ref{condition1} and \ref{condition1add} hold. Then we define the bijection $\rho: \mathbb{N}^+ \rightarrow \mathbb{N}^+$ such that $\lambda_{\rho(j)} = -\lambda_j$ and we can write
\beas
\big| \Im(\mathrm{Tr}(M(\lambda, D))) \big| &=& \f{1}{2}\left(\sum_{j=1}^{\infty} \f{|\lambda''| \beta_j}{(\lambda'-\lambda_j)^2+\lambda''^2} + \sum_{j=1}^{\infty} \f{|\lambda''| \beta^{(\rho(j))}}{(\lambda'+\lambda_j)^2+\lambda''^2} \right)\\
&=& \f{|\lambda''|}{2}\sum_{j=1}^{\infty} \f{(\lambda'^2+\lambda''^2+\lambda_j^2)(\beta^{(j)}+\beta^{(\rho(j))}) + 2\lambda'\lambda_j(\beta^{(j)}-\beta^{(\rho(j))})}{\big((\lambda'-\lambda_j)^2+\lambda''^2\big)\big((\lambda'+\lambda_j)^2+\lambda''^2\big)},
%
\eeas
where $\ds \beta_j = \sum_{l=1}^d\alpha^{(j)}_{l,l}$.

From Lemma \ref{lem-alpha>0} it follows that
\beas
\f{(\lambda'^2+\lambda''^2+\lambda_j^2)(\beta^{(j)}+\beta^{(\rho(j))}) + 2\lambda'\lambda_j(\beta^{(j)}-\beta^{(\rho(j))})}{\big((\lambda'-\lambda_j)^2+\lambda''^2\big)\big((\lambda'+\lambda_j)^2+\lambda''^2\big)}\geq0.
\eeas
Moreover,
\begin{align*}
\f{(\lambda'^2+\lambda''^2+\lambda_j^2)(\beta^{(j)}+\beta^{(\rho(j))}) + 2\lambda'\lambda_j(\beta^{(j)}-\beta^{(\rho(j))})}{\big((\lambda'-\lambda_j)^2+\lambda''^2\big)\big((\lambda'+\lambda_j)^2+\lambda''^2\big)} &\leq \\
&\f{(\lambda'^2+\lambda''^2+\lambda_j^2)(\beta^{(j)}+\beta^{(\rho(j))}) + 2\lambda'\lambda_j(\beta^{(j)}-\beta^{(\rho(j))})}{\lambda''^2(4\lambda'^2+\lambda''^2)}\\
&\quad+O(\frac{\lambda''^2}{4 \lambda'^2 + \lambda''^2}).
\end{align*}
Hence,
\beas
\big| \Im(\mathrm{Tr}(M(\lambda, D)))\big| \leq \f{|\lambda''|}{2}\sum_{j=1}^{\infty} \f{(\lambda'^2+\lambda''^2+\lambda_j^2)(\beta^{(j)}+\beta^{(\rho(j))}) + 2\lambda'(\lambda_j\beta^{(j)}+\lambda_{\rho(j)}\beta^{(\rho(j))})}{\lambda''^2(4\lambda'^2+\lambda''^2)} + O(\frac{\lambda''^2}{4 \lambda'^2 + \lambda''^2}).
\eeas
Using Lemma \ref{lem-sumRules} we obtain the following result. 
\begin{thm} Let $M(\lambda,D)$ be the polarization tensor associated with a $\mathcal{C}^{1,\alpha}$ domain $D\Subset \mathbb{R}^d$ with $\lambda = \lambda' + i\lambda''$ such that $|\lambda''|\ll 1$ and $|\lambda'|<1/2$. Then,
\begin{align*}
\big| \Im(\mathrm{Tr}(M(\lambda, D))) \big| &\leq \f{d |\lambda''| |D|}{\lambda''^2+4\lambda'^2} \\
& + \f{1}{|\lambda''|(\lambda''^2+4\lambda'^2)}\left(d\lambda'^2|D| +\f{(d-4)}{4}|D|+\sum_{l=1}^d\int_{ D} |\nabla \mathcal{S}_{D}[\nu_l]|^2dx + 2\lambda'\f{(d-2)}{2}|D| \right) \\ &\quad  +O(\frac{\lambda''^2}{4 \lambda'^2 + \lambda''^2}).
\end{align*}
\end{thm}
The bound in the above theorem depends not only on the volume of the particle but also on its geometry. Nevertheless, we remark that, since $|\lambda_j|<\f{1}{2}$,
\beas
\sum_{j=1}^{\infty}\lambda_j^2\sum_{l=1}^d\alpha^{(j)}_{l,l} < \f{d|D|}{4}.
\eeas
Hence, we can find a geometry independent, but not optimal, bound.
\begin{cor} We have
\begin{equation} \label{bound1}
\big| \Im(\mathrm{Tr}(M(\lambda, D))) \big| \leq  \f{1}{|\lambda''|(\lambda''^2+4\lambda'^2)}\left(d|D|\big(\lambda'^2 + \f{1}{4}\big) + 2\lambda'\f{(d-2)}{2}|D| \right) + \f{d |\lambda''| |D|}{\lambda''^2+4\lambda'^2}+O(\frac{\lambda''^2}{4 \lambda'^2 + \lambda''^2}).
\end{equation}
\end{cor}

\subsubsection{Bound for ellipses}
If $D$ is an ellipse whose semi-axes are on the $x_1$- and $x_2$- axes and of length $a$ and $b$, respectively, then its polarization tensor takes the form \cite{book3}
\begin{equation} \label{num2}
M(\lambda,D) = \left( \begin{array}{cc}
\ds \f{|D|}{\lambda-\f{1}{2}\f{a-b}{a+b}} & \ds 0 \\
\nm
\ds 0 & \ds \f{|D|}{\lambda + \f{1}{2}\f{a-b}{a+b}}
\end{array} \right).
\end{equation}
On the other hand, it is known that in $\mathcal{H}^* (\partial D)$ \cite{shapiro}
$$
\sigma(\mathcal{K}_D^*)\backslash\{1/2\} = \left\{\pm\f{1}{2}\left(\f{a-b}{a+b}\right)^j, \quad j=1,2,\ldots \right\}.
$$
Then, from \eqref{eq-PT}, we also have
\begin{equation*}
M(\lambda, D) = \left( \begin{array}{cc}
\ds \sum_{j=1}^{\infty} \f{\alpha^{(j)}_{1,1}}{\lambda-\f{1}{2}\left(\f{a-b}{a+b}\right)^j} & \ds \sum_{j=1}^{\infty} \f{\alpha^{(j)}_{1,2}}{\lambda-\f{1}{2}\left(\f{a-b}{a+b}\right)^j} \\ 
\ds \sum_{j=1}^{\infty} \f{\alpha^{(j)}_{1,2}}{\lambda-\f{1}{2}\left(\f{a-b}{a+b}\right)^j} & \ds \sum_{j=1}^{\infty} \f{\alpha^{(j)}_{2,2}}{\lambda-\f{1}{2}\left(\f{a-b}{a+b}\right)^j}
\end{array} \right).
\end{equation*}
Let $\ds \lambda_1 = \f{1}{2}\f{a-b}{a+b}$ and $\mathcal{V}(\lambda_j) = \{i\in\mathbb{N} \mbox{ such that } \mathcal{K}_D^*[\varphi_i] = \lambda_j\varphi_i\}$. It is clear now that \be
\label{add1} \sum_{i\in \mathcal{V}(\lambda_1)}\alpha^{(i)}_{1,1} = \sum_{i\in \mathcal{V}(-\lambda_{1})}\alpha^{(i)}_{2,2} = |D|, \quad \sum_{i\in \mathcal{V}(\lambda_j)}\alpha^{(i)}_{1,1} = \sum_{i\in \mathcal{V}(-\lambda_{j})}\alpha^{(i)}_{2,2} = 0\ee for $j\geq2$ and  $$\sum_{i\in \mathcal{V}(\lambda_j)}\alpha^{(i)}_{1,2} = 0$$ for $j\geq1$.

In view of (\ref{add1}), we have

\begin{align*}
\f{\beta^{(j)}}{(\lambda'-\lambda_j)^2+\lambda''^2} +  \f{\beta^{(\rho(j))}}{(\lambda'+\lambda_j)^2+\lambda''^2} \leq \f{4\lambda'^2\beta^{(j)} + \lambda''^2(\beta^{(j)} + \beta^{(j)})}{\lambda''^2(4\lambda'^2 + \lambda''^2)}+O(\frac{\lambda''^2}{4\lambda'^2 + \lambda''^2}).
\end{align*}

Hence, 
\beas
|\Im(\text{Tr}(M(\lambda, D)))| \leq \f{|\lambda''|}{2}\sum_{j=1}^{\infty} \f{4\lambda'^2\beta^{(j)}+ \lambda''^2(\beta^{(j)} + \beta^{(j)})}{\lambda''^2(4\lambda'^2 + \lambda''^2)} + O(\frac{\lambda''^2}{4\lambda'^2 + \lambda''^2}).
\eeas

Using Lemma \ref{lem-sumRules} we obtain the following result.
\begin{cor} For any ellipse $\widetilde{D}$ of semi-axes of length $a$ and $b$, 
we have 
\be\label{BoundUniv}
 |\Im(\text{Tr}(M(\lambda, \widetilde{D})))| \leq  \f{|\widetilde{D}| 4\lambda'^2}{|\lambda''|(\lambda''^2+4\lambda'^2)} + \f{2|\lambda''||\widetilde{D}|}{\lambda''^2+4\lambda'^2} + O(\frac{\lambda''^2}{4\lambda'^2 + \lambda''^2}).
\ee
\end{cor}
Figure \ref{figbound} shows (\ref{BoundUniv}) and the average extinction of two ellipses of semi-axis $a$ and $b$, where the ratio $a/b = 2$ and $a/b = 4$,  respectively.

\begin{figure}
\begin{center}
\includegraphics[scale=0.5]{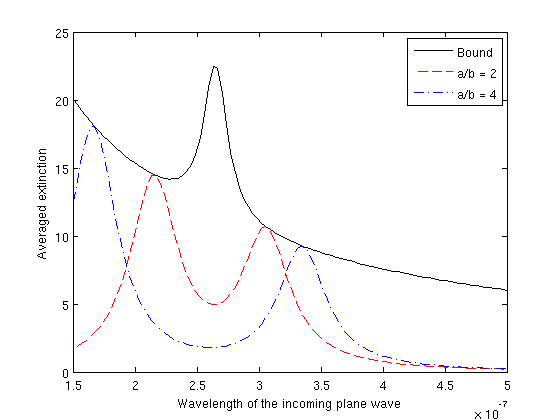}
\caption{\label{figbound} Optimal bound for ellipses.}
\end{center} 
\end{figure}


We can see from (\ref{num1}), Lemma \ref{lem-alpha>0} and the first sum rule in Lemma \ref{lem-sumRules} that for an arbitrary shape $B$, $|\Im(\text{Tr}(M(\lambda, B)))|$ is a convex combination of $\f{|\lambda''|}{(\lambda'-\lambda_j)^2+\lambda''^2}$ for $\lambda_j \in \sigma(\mathcal{K}_B^*)\backslash\{1/2\}$. Since ellipses put all the weight of this convex combination in $\pm \lambda_1 = \pm \f{1}{2}\f{a-b}{a+b}$, we have for any ellipse $\widetilde{D}$ and any shape $B$ such that $|B| = |\widetilde{D}|$,
\beas
|\Im(\text{Tr}(M(\lambda^*, B)))| \leq |\Im(\text{Tr}(M(\lambda^*, \widetilde{D})))|
\eeas
with $\lambda^* = \pm \f{1}{2}\f{a-b}{a+b} + i\lambda''$.

Thus, bound \eqref{BoundUniv} applies for any arbitrary shape $B$ in dimension two. This implies that, for a given material and a given desired 
resonance frequency  $\om^*$, the optimal shape for the extinction resonance (in the quasi-static limit) is an ellipse of semi-axis $a$ and $b$ such that $\lambda'(\om^*) = \pm\f{1}{2}\f{a-b}{a+b}$.

\subsubsection{Bound for ellipsoids}
Let $D$ be an ellipsoid given by
\be \label{eq-ellipsoid}
\f{x^2_1}{p^2_1} + \f{x^2_2}{p^2_2} + \f{x^2_3}{p^2_3} = 1.
\ee
The following holds \cite{book3}. 
\begin{lem} Let $D$ be the ellipsoid defined by \eqref{eq-ellipsoid}. Then, for $x\in D$, 
\beas
\mathcal{S}_D[\nu_l](x) = s_l  x_l, \quad l = 1,2,3,
\eeas
where
\beas
s_l = -\f{p_1p_2p_3}{2}\int_0^{\infty} \f{1}{(p^2_l +s)\sqrt{(p^2_1+s)(p^2_2+s)(p^2_3+s)}}ds.
\eeas
\end{lem}
Then we have 
\beas
\sum_{l=1}^3\int_{ D} |\nabla \mathcal{S}_{D}[\nu_l]|^2dx = (s_1^2+s_2^2+s_3^2)|D|.
\eeas
For a rotated ellipsoid $\widetilde{D} = \mathcal{R}D$ with $\mathcal{R}$ being a rotation matrix, we have $M(\lambda, \widetilde{D}) = \mathcal{R}M(\lambda, D)\mathcal{R}^T$ and so $\mathrm{Tr}(M(\lambda, \widetilde{D})) = \mathrm{Tr}(M(\lambda, D))$. Therefore, for any ellipsoid $\widetilde{D}$ of semi-axes of length $p_1, p_2$ and $p_3$ the following result holds.
\begin{cor} For any ellipsoid $\widetilde{D}$ of semi-axes of length $p_1, p_2$ and $p_3$, we have
\begin{equation} \label{bound3}
\Im(\mathrm{Tr}(M(\lambda, \widetilde{D}))) \leq
 \f{|\widetilde{D}|\left(3\lambda'^2 + \lambda' -\f{1}{4}+ (s_1^2+s_2^2+s_3^2) \right)}{|\lambda''|(\lambda''^2+4\lambda'^2)} + \f{3|\lambda''||\widetilde{D}|}{\lambda''^2+4\lambda'^2}
+O(\frac{\lambda''^2}{4\lambda'^2 + \lambda''^2}),
\end{equation}
where for $j=1,2,3$, 
$$
s_j = -\f{p_1p_2p_3}{2}\int_0^{\infty} \f{1}{(p^2_j+s)\sqrt{(p^2_1+s)(p^2_2+s)(p^2_3+s)}}ds.
$$
\end{cor}

\section{Link with the scattering coefficients} \label{sect4}

Our aim in this section is to exhibit the mechanism underlying  plasmonic resonances in terms of the scattering coefficients corresponding to the nanoparticle. The concept of scattering coefficients was first introduced in 
\cite{lim}. It plays a key role in constructing cloaking structures. It was extended in \cite{lim2} to the full Maxwell equations.  The scattering coefficients are simply the Fourier coefficients of the scattering amplitude
$A_\infty$. 
In Theorem \ref{thm5} we provide  an asymptotic expansion of the scattering amplitude in terms of the scattering coefficients of order $\pm1$. Our formula shows that, under physical conditions,  the scattering coefficients of orders $\pm1$ are the only scattering coefficients inducing the scattering cross-section enhancement.   For simplicity we only consider here the two-dimensional case. 

\subsection{The notion of scattering coefficients}

From Graf's addition formula \cite{book3} and \eqref{Helm-solution} the following asymptotic formula holds as $|x|\rightarrow \infty$ 
\beas
u^s(x) = (u-u^i)(x) = -\dfrac{i}{4}\sum_{n\in\mathbb{Z}}H_n^{(1)}(k_m|x|)e^{in\theta _x}\int_{\partial D}J_n(k_m|y|)e^{-in\theta _y}\psi(y)d\sigma(y),
\eeas
where $x = (|x|,\theta_x)$ in polar coordinates, $H_n^{(1)}$ is the Hankel function of the first kind and order $n$, $J_n$ is the Bessel function of order $n$ and $\psi$ is the solution to \eqref{equ-single}.\\
For $u^i (x)= e^{ik_md\cdot x}$ we have
\beas
u^i(x) = \sum_{m\in\mathbb{Z}}a_m(u^i)J_m(k_m|x|)e^{im\theta _x},
\eeas
where $a_m(u^i) = e^{im(\frac{\pi}{2}-\theta_d)}$.
By the superposition principle, we get
\beas
\psi = \sum_{m\in\mathbb{Z}}a_m(u^i)\psi_m,
\eeas
where $\psi_m$ is solution to \eqref{equ-single} replacing $f$ by
\beas
f^{(m)} := F_2^{(m)} + \f{1}{\mu_c} \big( \f{1}{2}Id - (\mathcal{K}_D^{k_c})^*\big)  (\mathcal{S}_D^{k_c})^{-1}[F_1^{(m)}]
\eeas
with
\beas
F_1^{(m)}(x) &=& -J_m(k_m|x|)e^{im\theta _x} ,\\
F_2^{(m)}(x) &=& -\f{1}{\mu_m}\f{\p J_m(k_m|x|)e^{im\theta _x}}{\p \nu}.
\eeas
We have
\beas
u^s(x) = (u-u^i)(x) = -\dfrac{i}{4}\sum_{n\in\mathbb{Z}}H_n^{(1)}(k_m|x|)e^{in\theta _x}\sum_{m\in\mathbb{Z}}W_{nm}e^{im(\frac{\pi}{2}-\theta_d)},
\eeas
where
\be \label{defscc}
W_{nm} = \int_{\partial D}J_n(k_m|y|)e^{-in\theta _y}\psi_m(y)d\sigma(y). 
\ee
The coefficients $W_{nm}$ are called the scattering coefficients.
\begin{lem} \label{lem-scatteringf}
In the space $\mathcal{H}^*(\partial D)$, as $\omega$ goes to zero, we have
\beas
f^{(0)} &=& O(\om^2),\\
f^{(\pm1)} &=& \om f_1^{(\pm1)} + O(\om^2),\\
f^{(m)} &=& O(\om^{m}), \quad |m| >1, 
\eeas
where 
\beas
f_1^{(\pm1)} = \mp\f{\sqrt{\eps_m\mu_m}}{2}\Big(\f{1}{\mu_m}e^{i\pm\theta _{\nu}} + \f{1}{\mu_c}(\dfrac{1}{2}Id - \mathcal{K}_{D}^*) \widetilde{\mathcal{S}}_{D}^{-1}[|x|e^{i\pm\theta _x}] \Big).
\eeas
\end{lem}
\begin{proof} Recall that $J_0(x) = 1 +O(x^2)$. By virtue of the fact that $$\big( \f{1}{2}Id - (\mathcal{K}_D^{k_c})^*\big)  (\mathcal{S}_D^{k_c})^{-1}[\chi(\p D)] = O(\om^2),$$ we arrive at the estimate for $f^{(0)}$ (see Appendix \ref{appen2d}). Moreover, 
$$J_{\pm1}(x) = \pm\f{x}{2} + O(x^3)$$ together with the fact that $$\big( \f{1}{2}Id - (\mathcal{K}_D^{k_c})^*\big)  (\mathcal{S}_D^{k_c})^{-1} = (\dfrac{1}{2}Id - \mathcal{K}_{D}^*) \widetilde{\mathcal{S}}_{D}^{-1}+O(\om^2\log\om)$$ gives the expansion of $f^{(\pm1)}$ in terms of $\omega$ (see Appendix \ref{appen2d}).

Finally, $J_m(x) = O(x^m)$ immediately yields the desired estimate for $f^{(m)}$.
\end{proof}
From Theorem \ref{thm12d}, it is easy to see that
\be \label{eq-psiScattering}
\psi_m = \sum_{j \in J} \f{\big( f^{(m)}, \widetilde{\varphi}_j(\om)\big)_{\mathcal{H}^*} \varphi_j(\om) }{ \tau_j(\om)}
+  {\mathcal{A}_{D}}(\om)^{-1}  ( P_{J^c}(\om) f).
\ee
Hence, from the definition of the scattering coefficients,
\be \label{eq-scatteringCoeffgeneral}
W_{nm} = \sum_{j \in J} \f{\big( f^{(m)}, \widetilde{\varphi}_j(\om)\big)_{\mathcal{H}^*} \Big(\varphi_j(\om) , J_n(k_m|x|)e^{-in\theta _x}\Big)_{-\f{1}{2},\f{1}{2}} }{ \tau_j(\om)}
+  \int_{\partial D}J_n(k_m|y|)e^{-in\theta _y}O(\om)d\sigma(y).
\ee
Since
\beas
J_m(x)\sim\f{1}{\sqrt(2\pi|m|)}\Big(\f{ex}{2|m|}\Big)^{|m|}
\eeas
as $m\rightarrow\infty$, we have
\beas
|f^{(m)}|\leq \f{C^{|m|}}{|m|^{|m|}}.
\eeas
Using the Cauchy-Schwarz inequality and Lemma \ref{lem-scatteringf}, we obtain the following result. 
\begin{prop} For $|n|,|m|>0$, we have
\beas
|W_{nm}|\leq \f{O(\om^{|n|+|m|})}{\min_{j\in J}|\tau_j(\om)|}\f{C^{|n|+|m|}}{|n|^{|n|}|m|^{|m|}}
\eeas
for a positive constant $C$ independent of $\omega$. 
\end{prop}

\subsection{The leading-order term in the expansion of the scattering amplitude}
In the following, we analyze the first-order scattering coefficients.
\begin{lem} \label{lem-psi_11} Assume that Conditions $1$ and $2$ hold. Then,
\beas
\psi_0 &=& \sum_{j \in J} \f{O(\om^2)}{ \tau_j(\om)}
+  O(\om),\\
\psi_{\pm1} &=& \sum_{j \in J}\f{ \pm\om \f{\sqrt{\eps_m\mu_m}}{2}\Big(\f{1}{\mu_m}-\f{1}{\mu_c}\Big)( e^{\pm i\theta _{\nu}},\varphi_j)_{\mathcal{H}^*} \varphi_j +O(\om^3\log \om)}{ \tau_j(\om)}
+  O(\om).\\
\eeas
\end{lem}
\begin{proof} The expression of 
$\psi_0$ follows from \eqref{eq-psiScattering} and Lemma \ref{lem-scatteringf}.
Changing $\mathcal{S}_D$ by $\widetilde{\mathcal{S}}_D$  in Theorem \ref{thm1} gives $\Big( (\dfrac{1}{2}Id - \mathcal{K}_{D}^*) \widetilde{\mathcal{S}}_{D}^{-1}[|x|e^{i\theta _x}], \varphi_j\Big)_{\mathcal{H}^*} = -( e^{i\theta _{\nu}} ,\varphi_j)_{\mathcal{H}^*}$. Using now Lemma \ref{lem-scatteringf} in \eqref{eq-psiScattering} yields the expression of $\psi_{\pm1}$.
\end{proof}
Recall that in two dimensions, 
$$
 \tau_j(\om) = \f{1}{2\mu_m} +  \f{1}{2\mu_c} - \big( \f{1}{\mu_c} - \f{1}{\mu_m} \big) \lambda_j + O(\om^2\log\om),
$$
where $\lambda_j$ is an eigenvalue of $\mathcal{K}^*_D$ and $\lambda_0 = {1}/{2}$. Recall also that for $0\in J$ we need $\tau_j\rightarrow 0$ and so $\mu_m\rightarrow \infty$, which is a limiting case that we can ignore. In practice, $P_J(\om)[\varphi_0(\om)] = 0$. We also have $(\varphi_j,\chi(\p D))_{-\f{1}{2},\f{1}{2}} = 0$ for $j \neq 0$.\\
It follows then from the above lemmas and the expression \eqref{eq-scatteringCoeffgeneral} of the scattering coefficients that
\beas
W_{00} &=& \sum_{j \in J} \f{O(\om^4\log\om)}{ \tau_j(\om)}
+  O(\om),\\
W_{0\pm1} &=& \sum_{j \in J} \f{O(\om^3\log\om)}{ \tau_j(\om)}
+  O(\om),\\
W_{\pm10} &=& \sum_{j \in J} \f{O(\om^3) }{ \tau_j(\om)}
+  O(\om^2).
\eeas
Note that $W_{\pm1\pm1}$ has a special structure. Indeed, from Lemma \ref{lem-psi_11} and equation \eqref{eq-scatteringCoeffgeneral},  we have
\beas
W_{\pm1\pm1} &=& \sum_{j \in J} \f{\pm\pm\om \f{\sqrt{\eps_m\mu_m}}{2}\Big(\f{1}{\mu_m}-\f{1}{\mu_c}\Big)\big(\varphi_j , J_1(k_m|x|)e^{\mp i\theta _x}\big)_{-\f{1}{2},\f{1}{2}}\big( e^{\pm i\theta _{\nu}}, \varphi_j\big)_{\mathcal{H}^*} + O(\om^4\log \om) }{ \tau_j(\om)}
+  O(\om^2),\\
 &=& \sum_{j \in J} \f{\pm\pm\om^2 \f{\eps_m\mu_m}{4}\Big(\f{1}{\mu_m}-\f{1}{\mu_c}\Big)\big(\varphi_j , |x|e^{\mp i\theta _x} \big)_{-\f{1}{2},\f{1}{2}}\big( e^{\pm i\theta _{\nu}}, \varphi_j\big)_{\mathcal{H}^*}  + O(\om^4\log \om) }{ \tau_j(\om)}
+  O(\om^2),\\
&=&\f{k_m^2}{4}\left(\sum_{j \in J} \f{\pm\pm\big(\varphi_j , |x|e^{\mp i\theta _x} \big)_{-\f{1}{2},\f{1}{2}}\big( e^{\pm i\theta _{\nu}}, \varphi_j\big)_{\mathcal{H}^*} + O(\om^2\log \om) }{ \lambda-\lambda_j+O(\om^2\log\om)}
+  O(1)\right),
\eeas
where $\lambda$ is defined by (\ref{deflambda}). Now, assume that $\min_{j\in J} |\tau_j(\om)| \gg \om^2\log\om$. Then,
\be \label{eq-W11}
W_{\pm1\pm1} = \f{k_m^2}{4}\left(\sum_{j \in J} \f{\pm\pm \big(\varphi_j , |x|e^{\mp i\theta _x} \big)_{-\f{1}{2},\f{1}{2}}\big( e^{\pm i\theta _{\nu}}, \varphi_j\big)_{\mathcal{H}^*} }{ \lambda-\lambda_j}
+  O(1)\right).
\ee
Define the contracted polarization tensors by
\beas
N_{\pm,\pm}(\lambda, D) := \int_{\partial D} |x|e^{\pm i\theta_x} (\lambda I - \mathcal{K}^*_{D})^{-1}[e^{\pm i\theta _{\nu}}](x)\,  d\sigma(x).
\eeas
It is clear that
\beas
N_{+,+}(\lambda, D) &=& M_{1,1}(\lambda, D) - M_{2,2}(\lambda, D) + i2M_{1,2}(\lambda, D),\\
N_{+,-}(\lambda, D) &=& M_{1,1}(\lambda, D) + M_{2,2}(\lambda, D),\\
N_{-,+}(\lambda, D) &=& M_{1,1}(\lambda, D) + M_{2,2}(\lambda, D),\\
N_{-,-}(\lambda, D) &=& M_{1,1}(\lambda, D) - M_{2,2}(\lambda, D) - i2M_{1,2}(\lambda, D),
\eeas
where $M_{l,m}(\lambda, D)$ is the $(l,m)$-entry of the polarization tensor given by (\ref{defm}). 

Finally,  considering the above we can state the following result. 
\begin{thm} \label{thm5} Let $A_\infty$ be the scattering amplitude in the far-field defined in (\ref{eq-scatteringAmplitude}) for the incoming plane wave $u^i (x)= e^{ik_m d\cdot x}$. Assume Conditions $1$ and $2$ and
$$
\min_{j\in J} |\tau_j(\om)| \gg \om^2\log\om.
$$
Then, $A_\infty$ admits the following asymptotic expansion
\beas
A_\infty\left(\frac{x}{\vert x \vert}\right) = \frac{x}{\vert x \vert}^T W_1 d + O(\om^2),
\eeas
where
$$
W_1 = \left( \begin{array}{cc}
W_{-11}+W_{1-1}-2W_{1,1} & i\big( W_{1-1} - W_{-11} \big) \\
i\big( W_{1-1} - W_{-11} \big) & -W_{-11}-W_{1-1}-2W_{11}
\end{array} \right).
$$
Here, $W_{nm}$ are the scattering coefficients defined by (\ref{defscc}). 
\end{thm}
\begin{proof}
From \eqref{eq-scatteringAmplitude2}, we have
\beas
A_\infty\left(\frac{x}{\vert x \vert}\right) =-k_m^2 \frac{x}{\vert x \vert}^T M(\lambda, D) d.
\eeas
Since $\mathcal{K}_D^*$ is compact and self-adjoint in $\mathcal{H}^*$, we have
\beas
N_{\pm,\pm}(\lambda, D) &=& \sum_{j = 1}^{\infty}\f{\big(\varphi_j , |x|e^{\pm i\theta _x} \big)_{-\f{1}{2},\f{1}{2}}\big( e^{\pm i\theta _{\nu}}, \varphi_j\big)_{\mathcal{H}^*}  }{ \lambda-\lambda_j}\\
&=& \sum_{j\in J}\f{\big(\varphi_j , |x|e^{\pm i\theta _x} \big)_{-\f{1}{2},\f{1}{2}}\big( e^{\pm i\theta _{\nu}}, \varphi_j\big)_{\mathcal{H}^*} }{ \lambda-\lambda_j} + O(1).
\eeas
We have then from \eqref{eq-W11} that
\beas
-\f{k_m^2}{4}N_{+,+}(\lambda, D) &=&  W_{-11} + O(\om^2),\\
-\f{k_m^2}{4}N_{+,-}(\lambda, D) &=& -W_{11} + O(\om^2),\\
-\f{k_m^2}{4}N_{-,+}(\lambda, D) &=& -W_{11} + O(\om^2), \\
-\f{k_m^2}{4}N_{-,-}(\lambda, D) &=& W_{1-1} + O(\om^2).
\eeas
In view of 
\beas
M_{11} &=& \f{1}{4}\left( N_{+,+} + N_{-,-} + 2N_{+,-}\right),\\
M_{22} &=& \f{1}{4}\left( -N_{+,+} - N_{-,-} + 2N_{+,-} \right),\\
M_{12} &=& \f{-i}{4}\left( N_{+,+}-N_{-,-} \right),
\eeas
we get the result.
\end{proof}

\section{Super-resolution (super-focusing) by using plasmonic particles}
\label{sect5}

It is known that the resolution limit (or the diffraction limit) in a general inhomogeneous  space is determined by the imaginary part of the Green function in the associated space \cite{book1}. By modifying the homogeneous spaces with subwavelength resonators, we can introduce propagating subwavelength resonance
modes to the space which encode subwavelength information in a neighborhood of the space embedded by the subwavelenghth resonators, thus yield a Green's function whose imaginary part exhibits  subwavelength peaks and therefore break the resolution limit (or diffraction limit) in the homogeneous space. The principle has been mathematically demonstrated  in \cite{hai}. Here, using the fact that plasmonic particles are ideal subwavelength resonators, we consider the possibility of super-resolution (super-focusing) by using a system of identical plasmonic particles. The results in this section can be viewed as a consequence of the results in Section \ref{sec-multi-scatter}.
\subsection{Asymptotic expansion of the scattered field}
In order to illustrate the superfocusing phenomenon, we set 
$$
u^i (x)= G(x, x_0, k_m)= -\f{e^{ik_m|x-x_0|}}{4 \pi |x-x_0|}.
$$

\begin{lem}
In the space $\mathcal{H}^*(\p D)$, as $\omega$ goes to zero, we have
$$
f = f_0 + O(\om\delta^{\f{3}{2}}) + O(\delta^{\f{5}{2}}),
$$
where $f_0=(f_{0,1},\ldots, f_{0, L})^T$ with
$$
f_{0,l}= - \f{1}{4\pi|z_l -x_0|^3} \left( \f{1}{\mu_m}(z_l-x_0)\cdot \nu(x)+
  \f{1}{\mu_c} (\f{1}{2}Id - \mathcal{K}_{D_l}^*) \mathcal{S}_{D_l}^{-1} [(z_l-x_0)\cdot(x-z_l)] \right)
  =O(\delta^{\f{3}{2}}).
$$
\end{lem}
\begin{proof}
The proof is similar to that of Lemma \ref{lem-f-1}.
Recall that
$$
f_l  = F_{l,2} + \f{1}{\mu_c} \big( \f{1}{2}Id - (\mathcal{K}_{D_l}^{k_c})^*\big)  (\mathcal{S}_{D_l}^{k_c})^{-1}[F_{l,1}].
$$
We can show that
$$
F_{l,2}= -\f{1}{\mu_m} \f{\p u^i}{\p \nu} = -\f{1}{4\pi \mu_m |z_l -x_0|^3}(z_l-x_0)\cdot \nu(x) + O(\delta^{\f{5}{2}})+ O(\om\delta^{\f{3}{2}}) \quad \mbox{in}\,\, \mathcal{H}^*(\p D_l).
$$
Besides,
$$
u^i(x)|_{\p D_l} = - \f{e^{ik_m|z_l-x_0|}}{4 \pi |z_l-x_0|} \chi(\p D_l) + \f{1}{4\pi|z_l -x_0|^3}(z_l-x_0)\cdot (x-z_l) + O(\delta^{\f{5}{2}})+ O(\om \delta^{\f{3}{2}}) \quad \mbox{in}\,\, \mathcal{H}(\p D_l).
$$
Using the identity $(\f{1}{2}Id - \mathcal{K}_{D_l}^*) \mathcal{S}_{D_l}^{-1}[\chi(\p D_l)]=0$, we obtain that
$$
\f{1}{\mu_c} \big( \f{1}{2}Id - (\mathcal{K}_{D_l}^{k_c})^*\big)  (\mathcal{S}_{D_l}^{k_c})^{-1}[F_{l,1}]
=- \f{1}{4\pi|z_l -x_0|^3\mu_c} (\f{1}{2}Id - \mathcal{K}_{D_l}^*) \mathcal{S}_{D_l}^{-1} [(z_l-x_0)\cdot(x-z_l)].
$$
This completes the proof of the lemma.
\end{proof}

We now derive an asymptotic expansion of the scattered field in an intermediate regime which is neither too close to the plasmonic particles nor too far away. More precisely, we consider the following domain
$$
\mathrm{D}_{\delta, k} = \big\{x \in \R^3; \min_{1\leq l \leq L} |x-z_l| \gg \delta, \,\, 
\max_{1\leq l \leq L} |x-z_l| \ll \f{1}{k} \big\}.
$$

\begin{lem}  \label{lem-field-expansion}
Let $\psi_l \in \mathcal{H}^*(\p D_l)$ and let $v(x) = \mathcal{S}_{D_l}^{k} [\psi_l](x)$. Then we have for $x\in \mathrm{D}_{\delta, k}$, 
\beas \label{eq:farfieldexpansion2}
v(x) &=& G(x, z_l, k)\big(\f{1}{|x-z_l|} - ik \big) \f{x-z_l}{|x-z_l|}\cdot 
\int_{\partial D_0} y \psi_l(y) d\sigma(y) + O(\delta^{\f{5}{2}})
\|\psi_l \|_{\mathcal{H}^*(\p D_l)} \\
&& +  G(x, z_l, k) \int_{\partial D_0} \psi_l(y) d\sigma(y).
\eeas
Moreover, the following estimates hold
\beas
v(x) &=& O(\delta^{\f{3}{2}}) \quad \mbox{ if }\,\, \int_{\partial D_0} \psi_l(y) d\sigma(y)=0, \\
v(x) &=& O(\delta^{\f{1}{2}}) \quad \mbox{ if }\,\, \int_{\partial D_0} \psi_l(y) d\sigma(y)\neq 0. 
\eeas
\end{lem}

\begin{proof} We only consider the case when $l=0$. The other case follows similarly or by coordinate translation.   
We have
$$
v(x)=\mathcal{S}_D^{k}[\psi](x)=  \int_{\partial D_0} G(x, y, k)\psi(y) d\sigma(y)
 = - \int_{\partial D_0} \frac{e^{ik \vert x-y\vert}}{4\pi \vert x-y\vert} \psi(y) d\sigma(y) .
$$
Since
$$
G(x, y, k) = G(x, 0, k) + \sum_{|\alpha=1|} \f{\p G(x, 0, k)}{\p y^{\alpha}}y^{\alpha} + 
\sum_{m \geq 2}\sum_{|\alpha=m|} \f{\p^m G(x, 0, k)}{\p y^{\alpha}}y^{\alpha},
$$
and
$$
\f{\p G(x, 0, k)}{\p y^{\alpha}}= -\frac{e^{ik \vert x\vert}}{4\pi \vert x\vert}\big(\f{1}{|x|} - ik \big) \f{x}{|x|}
= G(x, 0, k)\big(\f{1}{|x|} - ik \big) \f{x^{\alpha}}{|x|},
$$
we obtain the required identity for the case $l=0$. The estimate follows from the fact that
$$
\| y^{\alpha}\|_{\mathcal{H}(\partial D_0)} = O( \delta^{\f{2|\alpha|+1}{2}}). 
$$
This completes the proof of the lemma. \end{proof}

Denote by
\beas
S_{j,l}(x, k) &=& G(x, z_l, k)\f{x-z_l}{|x-z_l|^2} 
\cdot \int_{\partial D_0} y \varphi_j(y) d\sigma(y), \\
S_{l}(x, k) &=& G(x, z_l, k) \int_{\partial D_0} \varphi_0(y) d\sigma(y), \\
H_{j, l}(x_0)&=& - \f{1}{4 \pi |z_l-x_0|^3} 
\big( (z_l- x_0)\cdot \nu(x), \varphi_j\big)_{\mathcal{H}^*(\p D_0)}.
\eeas

It is clear that the following size estimates hold
\beas
S_{j,l}(x, k) = O(\delta^{\f{3}{2}}), \quad S_{l}(x, k) = O(\delta^{\f{1}{2}}), \quad H_{j,l}(x_0) = O(\delta^{\f{3}{2}}) \quad \mathrm{for } j \neq 0, \quad  H_{O,l}(x_0) = 0.
\eeas

\begin{thm} \label{thm2b}
Under Conditions \ref{condition0}, \ref{condition1}, \ref{condition1add},
and \ref{cond-multi}, the Green function $\Gamma(x, x_0, k_m)$ in the presence of $L$ plasmonic particles has the following representation in the quasi-static regime: 
for $x \in \mathrm{D}_{\delta, k_m}$,
\beas
\Gamma(x, x_0, k_m) &=& G(x, x_0, k_m) \\
&&+ \sum_{j \in J} \sum_{l=1}^L \f{H_{j, p}(x_0)
{\widetilde{X}_{j, l, p}}  X_{j, l,q} S_{j,q}(x,k_m)  + O(\delta^4)+ O(\omega \delta^3)  }
{\lambda - \lambda_j + \big( \f{1}{\mu_c}-\f{1}{\mu_m}  \big)^{-1}\tau_{j, l} + O(\delta^4)+ O(\delta^2 \omega^2) } + O(\delta^3). 
\eeas
\end{thm}

\begin{proof}
With $u^i (x)= G(x, x_0, k_m)$, we have
$$
\psi= \sum_{j \in J} \sum_{1 \leq l \leq L} a_{j, l} \varphi_{j,l} + \sum_{1 \leq l \leq L} a_{0, l} \varphi_{0,l}
+ O( \delta^{\f{3}{2}}),
$$
where 
\beas
a_{j, l} &=&
 (f, \widetilde{\varphi}_{j, l})_{\mathcal{H}^*(\p D)} = (f_0, \widetilde{\varphi}_{j, l})_{\mathcal{H}^*(\p D)} + O(\om\delta^{\f{3}{2}}) + O(\delta^{\f{5}{2}}),\\
&=&  (\f{1}{\mu_c} -\f{1}{\mu_m} ) {\widetilde{X}_{j,l, p}} H_{j,p}(x_0) + O(\om\delta^{\f{3}{2}}) + O(\delta^{\f{5}{2}}),\\
a_{0, l} &=& (f, \widetilde{\varphi}_{0, l})_{\mathcal{H}^*(\p D)} =  O(\delta^{\f{5}{2}}).
\eeas

By Lemma \ref{lem-field-expansion}, 
\beas
\mathcal{S}_D^{k_m}[\varphi_{j,l}](x) &=& \sum_{1\leq p \leq L} \mathcal{S}_{D}^{k_m}[ X_{j,l, p} \varphi_j e_p](x)
=  \sum_{1\leq p \leq L} X_{j,l, p} \mathcal{S}_{D_p}^{k_m}[  \varphi_j](x) \\
&=& \sum_{1\leq p \leq L} X_{j,l, p} S_{j, p}(x, k_m) + O(\delta^{\f{5}{2}})+ O(\om \delta^{\f{3}{2}}).
\eeas

On the other hand, 
for $j = 0$, we have 
\beas
\mathcal{S}_D^{k_m}[\varphi_{0,l}](x) &=& O(\delta^{\f{1}{2}}), \\
\tau_{0, l}(\omega) &=& \tau_0 + O(\delta^4)+ O(\delta^2 \omega^2)=O(1).  
\eeas

Therefore, we can deduce that
\beas
u^s &=&  \mathcal{S}_D^{k_m}[\psi](x) = 
\sum_{j \in J} \sum_{1 \leq l \leq L} a_{j, l} \mathcal{S}_D^{k_m} [\varphi_{j,l}] + \sum_{1 \leq l \leq L} a_{0, l} \mathcal{S}_D^{k_m} [\varphi_{0,l}]
+ O(\delta^{3}) ,\\
&=&  \sum_{j \in J} \sum_{l=1}^L \f{1}{\tau_{j,l}(\om)} \Big( (\f{1}{\mu_c} -\f{1}{\mu_m})H_{j, p}(x_0)
{\widetilde{X}_{j, l, p}} X_{j, l,q} S_{j,q}(x,k_m)  + O(\omega \delta^3) +O(\delta^{4})  \Big) \\
&& + O(\delta^3 ) ,\\
&=& \sum_{j \in J} \sum_{l=1}^L \f{H_{j, p}(x_0)
{\widetilde{X}_{j, l, p}} X_{j, l,q} S_{j,q}(x,k_m)  + O(\omega \delta^3) +O(\delta^{4}) }
{\lambda - \lambda_j + \big( \f{1}{\mu_c}-\f{1}{\mu_m}  \big)^{-1}\tau_{j, l} + O(\delta^4)+ O(\delta^2 \omega^2) } + O(\delta^{3}).
\eeas

 
\end{proof}

\subsection{Asymptotic expansion of the imaginary part of the Green function}
As a consequence of Theorem \ref{thm2b}, we obtain the following result on the imaginary part of the Green function.

\begin{thm} \label{thm-super}
Assume the same conditions as in Theorem \ref{thm2b}. Under the additional assumption that
\beas
\lambda - \lambda_j + \big( \f{1}{\mu_c}-\f{1}{\mu_m}  \big)^{-1}\tau_{j, l} & \gg &  O(\delta^4)+ O(\delta^2 \omega^2), \\
\Re \left(\lambda - \lambda_j + \big( \f{1}{\mu_c}-\f{1}{\mu_m}  \big)^{-1}\tau_{j, l}\right) & \lesssim & 
 \Im\left(\lambda - \lambda_j + \big( \f{1}{\mu_c}-\f{1}{\mu_m}  \big)^{-1}\tau_{j, l}\right)
\eeas
for each $l$ and $j\in J$, 
we have
\beas
\Im{\Gamma(x, x_0, k_m)} &=& \Im{G(x, x_0, k_m)} + O(\delta^3) + \\
&& \sum_{j \in J} \sum_{l=1}^L \Re \left( H_{j, p}(x_0)
{\widetilde{X}_{j, l, p}} X_{j, l,q} S_{j,q}(x,0)  + O(\omega \delta^3) +O(\delta^{4})\right) \\
&& \times \Im\left(\f{1}{\lambda - \lambda_j + \big( \f{1}{\mu_c}-\f{1}{\mu_m} \big)^{-1}\tau_{j, l}}\right) ,
\eeas
where $x, x_0 \in \mathrm{D}_{\delta, k_m}$. 
\end{thm}

Note that $\Re \left(H_{j, p}(x_0) {\widetilde{X}_{j, l, p}} X_{j, l,q} S_{j,q}(x,0) \right)= O(\delta^3)$. 
Under the conditions in Theorem \ref{thm-super}, if we have additionally that
$$
\Im\left(\f{1}{\lambda - \lambda_j + \big( \f{1}{\mu_c}-\f{1}{\mu_m}  \big)^{-1}\tau_{j, l}}\right)
= O(\frac{1}{\delta^3})
$$
for some plasmonic frequency $\omega$, then the term in the expansion of $\Im{\Gamma(x, x_0, k_m)}$ which is due to resonance has size one and exhibits subwavelength peak with width of order one. This breaks the diffraction limit ${1}/{k_m}$ in the free space. We also note that the term 
$\Im{G(x, x_0, k_m)}$ has size $O(\omega)$. 
Thus, we can conclude that super-resolution (super-focusing) can indeed be achieved by using a system of plasmonic particles.  
%
%
\section{Concluding remarks}
In this paper, based on perturbation arguments, we studied the scattering by plasmonic nanoparticles when the frequency is close to a resonant frequency.  
We have shown that plasmon resonant nanoparticles provide a possible way not only of super-resolved imaging but also of scattering and absorption enhancements.

We have derived the shift and broadening of the plasmon resonance with changes in size. We have also consider the case of multiple nanoparticles under the weak interaction assumption. The localization algorithms developed in \cite{book3, iakovleva, gang} can be extended to the problem of imaging plasmonic nanoparticles. We have precisely quantified the scattering and absorption cross-section enhancements and gave optimal bounds on the enhancement factors. We have also linked the plasmonic resonances to the scattering coefficients and showed that the leading-order term of the scattering amplitude can be expressed in terms of the $\pm$-one order of the scattering coefficients. 
 
The generalization to the full Maxwell equations of the methods and results of the paper is under consideration and will be reported elsewhere. Another challenging problem is to optimize the super-focusing phenomenon in terms of the organization of the nanoparticles. This will be also the subject of a forthcoming publication.

\appendix
\section{Asymptotic expansion of the integral operators: single particle} \label{append1}

In this section, we derive asymptotic expansions with respect to $k$ of some boundary integral operators defined on the boundary of a bounded and simply connected smooth domain $D$ in dimension three whose size is of order one. 

We first consider the single layer potential
$$
\mathcal{S}_{D}^{k} [\psi](x) =  \int_{\p D} G(x, y, k) \psi(y) d\sigma(y),  \quad x \in  \p {D},
$$
where $$G(x, y, k)= - \f{e^{ik|x-y|}}{4 \pi|x-y|}$$ is the  Green function of Helmholtz equation in $\R^3$, subject to the Sommerfeld radiation condition.
Note that
$$
G(x, y, k) = - \sum_{j =0}^{\infty} \f{(ik|x-y|)^j}{j! 4 \pi |x-y|}
= - \f{1}{4 \pi |x-y|} - \f{ik}{4 \pi}\sum_{j =1}^{\infty} \f{ (ik|x-y|)^{j-1}}{j! }.
$$
We get
\be \label{series-s}
\mathcal{S}_{D}^{k}=  \mathcal{S}_{D} + \sum_{j=1}^{\infty} k^j \mathcal{S}_{D, j},
\ee
where
$$
\mathcal{S}_{D, j} [\psi](x) = - \f{i}{4 \pi} \int_{\p D} \f{ (i|x-y|)^{j-1}}{j! } \psi(y)d\sigma(y).
$$
In particular, we have
\bea
\mathcal{S}_{D, 1} [\psi](x) &=& - \f{i}{4 \pi}  \int_{\p D}  \psi(y)d\sigma(y), \\
\mathcal{S}_{D, 2} [\psi](x) &=& - \f{1}{4 \pi}  \int_{\p D}  |x-y| \psi(y)d\sigma(y).
\eea

\begin{lem} \label{lem-appendix11}
$\| \mathcal{S}_{D, j} \|_{
\mathcal{L}(({\mathcal{H}^*(\p D)}, \mathcal{H}(\p D))}$ is uniformly bounded with respect to $j$. Moreover, the series in (\ref{series-s}) is convergent in $\mathcal{L}({\mathcal{H}^*(\p D)}, \mathcal{H}(\p D))$.
\end{lem}
\begin{proof}
It is clear that
$$
\| \mathcal{S}_{D, j} \|_{ \mathcal{L}(L^2(\p D), H^{1}(\p D))} \leq C,
$$
where $C$ is independent of $j$.
On the other hand, a similar estimate also holds for the operator $\mathcal{S}^*_{D, j}$. It follows that
$$
\| \mathcal{S}_{D, j} \|_{ \mathcal{L}( {H^{-1}(\p D)}, L^2(\p D) )} \leq C.
$$
Thus, we can conclude that $\| \mathcal{S}_{D, j} \|_{\mathcal{L}(H^{-\f{1}{2}}(\p D), H^{\f{1}{2}}(\p D))}$ is uniformly bounded by using interpolation theory. By the equivalence of norms in
the $H^{-\f{1}{2}}(\p D)$ and $H^{\f{1}{2}}(\p D)$, the lemma follows immediately.
\end{proof}

Note that $\mathcal{S}_{D}$ is invertible in dimension three, so is $\mathcal{S}_{D}^{k}$ for small $k$.
By formally writing
\be  \label{series-b}
(\mathcal{S}_{D}^{k})^{-1} = \mathcal{S}_{D}^{-1} + k \mathcal{B}_{D, 1} + k^2 \mathcal{B}_{D, 2} + \ldots,
\ee
and using the identity $(\mathcal{S}_{D}^{k})^{-1} \mathcal{S}_{D}^{k} = Id$, we can derive that
\be \label{defB12}
\mathcal{B}_{D, 1} = - \mathcal{S}_D^{-1} \mathcal{S}_{D, 1}\mathcal{S}_D^{-1}, \quad
\mathcal{B}_{D, 2} = - \mathcal{S}_D^{-1} \mathcal{S}_{D, 2}\mathcal{S}_D^{-1}
+ \mathcal{S}_D^{-1}\mathcal{S}_{D, 1}\mathcal{S}_D^{-1} \mathcal{S}_{D, 1}\mathcal{S}_D^{-1}.
\ee
We can also derive other lower-order terms $\mathcal{B}_{D, j}$.

\begin{lem} \label{lem-appendix12}
The series in (\ref{series-b}) converges in  $\mathcal{L}(\mathcal{H}(\p D), \mathcal{H}^*(\p D))$ for sufficiently small $k$.
\end{lem}
\begin{proof}
The proof can be deduced from the identity
 $$ (\mathcal{S}_{D}^{k})^{-1} = (Id + \mathcal{S}_D^{-1} \sum_{j=1}^{\infty}k^j \mathcal{S}_{D,j})^{-1}\mathcal{S}_D^{-1}. $$
\end{proof}

We now consider the expansion for the boundary integral operator $(\mathcal{K}_{D}^{k})^*$. We have
\be \label{series-k}
(\mathcal{K}_{D}^{k})^*  = \mathcal{K}_D^* + k \mathcal{K}_{D, 1} + k^2 \mathcal{K}_{D, 2} + \ldots,
\ee
where
$$
\mathcal{K}_{D, j}[\psi](x) = - \f{i}{4 \pi} \int_{\p D} \f{ \p (i|x-y|)^{j-1}}{j! \p \nu(x)} \psi(y) d\sigma(y)=
- \f{i^j (j-1)}{4 \pi j!} \int_{\p D} |x-y|^{j-3} (x-y)\cdot\nu(x) \psi(y) d\sigma(y).
$$
In particular, we have
\be
\mathcal{K}_{D, 1} =0, \quad   \mathcal{K}_{D, 2}[\psi](x) = \f{1}{4\pi} \int_{\p D} \f{(x-y)\cdot \nu(x)}{|x-y|} \psi(y)d\sigma(y).
\ee

\begin{lem} \label{lem-appendix13} The norm
$\| \mathcal{K}_{D, j} \|_{ \mathcal{L}(\mathcal{H}^*(\p D), \mathcal{H}^*(\p D) )}$ is uniformly bounded for $j \geq 1$.
Moreover, the series in (\ref{series-k}) is convergent in $\mathcal{L}({\mathcal{H}^*(\p D)}, {\mathcal{H}^*(\p D)})$.
\end{lem}

\section{Asymptotic expansion of the integral operators: multiple particles} \label{append2}

In this section, we consider the three-dimensional case. We assume that the particles have size of order $\delta$ which is a small number and the distance between them is of order one. We write $D_j = z_j + \delta \widetilde{D}$, $j=1, 2, \ldots, M$, where $\widetilde{D}$ has size one and is centered at the origin. Our goal is to derive estimates for  various boundary integral operators considered in the paper that are defined on small particles in terms of their size.
For this purpose, we denote by $D_0 = \delta \widetilde{D}$. For each function $f$ defined on $\p D_0$, we define a corresponding function  on $\widetilde{D}$  by 
$$
\eta(f)(\widetilde{x}) = f(\delta \widetilde{x}).
$$

We first state some useful results.

\begin{lem} \label{lem-scaling-1}
The following scaling properties hold:
\begin{enumerate}
\item[(i)]
$\| \eta(f)\|_{L^2(\p \widetilde{D})} = {\delta}^{-1} \| f\|_{L^2(\p D_0)}$;
\item[(ii)]
$\| \eta(f)\|_{\mathcal{H}(\p \widetilde{D})} = {\delta^{- \f{1}{2}}} \| f\|_{\mathcal{H}(\p D_0)}$;
\item[(iii)]
$\| \eta(f)\|_{\mathcal{H}^*(\p \widetilde{D})} = {\delta^{-\f{3}{2}}} \| f\|_{\mathcal{H}^*(\p D_0)}$.
\end{enumerate}

\end{lem}

\begin{proof}
The proof of (i) is straightforward and we only need to prove (ii) and (iii).
To prove (iii), we have
\begin{eqnarray*}
\|f\|^2_{\mathcal{H}^*(\p D_0)} &=& \int_{\p D_0} \int_{\p D_0} \f{f(x) 
f(y)}{ 4\pi |x-y|}d\sigma(x)d\sigma(y)\\
&=& \delta^3 \int_{\p \widetilde{D}} \int_{\p \widetilde{D}} \f{\eta (f)(\widetilde{x}) \eta(f)(\widetilde{y})}{ 4 \pi |\widetilde{x}-\widetilde{y}|}d\sigma(\widetilde{x})d\sigma(\widetilde{x})\\
&=& \delta^3 \| \eta(f)\|^2_{\mathcal{H}^*(\p \widetilde{D})},
\end{eqnarray*}
whence (iii) follows.
To prove (ii), recall that
$$
\| f\|_{\mathcal{H}(\p D_0)} = \| \mathcal{S}_{D_0}^{-1}f\|_{\mathcal{H}^*(\p D_0)}.
$$
Let $u=\mathcal{S}_{ D_0}^{-1}[f] $. Then $f= \mathcal{S}_{D_0}[u]$. We can show that
$$
\eta(f) = \delta \mathcal{S}_{\widetilde{D}} (\eta (u)).
$$
As a result, we have
$$
\| \eta(f)\|_{\mathcal{H}(\p \widetilde{D})} = \delta \| \mathcal{S}_{ \widetilde{D}} (\eta (u))\|_{\mathcal{H}(\p \widetilde{D})}
= \delta \| \eta (u)\|_{\mathcal{H}^*(\p \widetilde{D})}
= \delta^{-\f{1}{2}} \|u\|_{\mathcal{H}^*(\p D_0)} = \delta^{-\f{1}{2}}\| f\|_{\mathcal{H}(\p D_0)},
$$
which proves (ii).
\end{proof}


\begin{lem} \label{lem-scaling}
Let $X$ and $Y$ be bounded and simply connected smooth domains in $\R^3$. Assume $0\in X, Y$ and $X= \delta \widetilde{X}$,
$Y= \delta \widetilde{Y}$.
Let $\mathcal{R}$ and $\widetilde{\mathcal{R}}$ be two boundary integral operators from $\mathcal{D}'(\p Y)$ to $\mathcal{D}'(\p X)$ and $\mathcal{D}'(\p \widetilde{Y})$ to $\mathcal{D}'(\p \widetilde{X})$, respectively. Here, $\mathcal{D}'$ denotes the Schwartz space. Assume that both operators have the same Schwartz kernel $R$ with the following homogeneous scaling property
$$
R(\delta x, \delta y) = \delta^m R(x, y).
$$
Then,
\begin{eqnarray*}
\| \mathcal{R}\|_{\mathcal{L}(\mathcal{H}^*(\p Y), \mathcal{H}^*(\p X))} &=& \delta^{2+m}\| \widetilde{\mathcal{R}}\|_{\mathcal{L}(\mathcal{H}^*(\p \widetilde{Y}), \mathcal{H}^*(\p \widetilde{X})) },\\
\| \mathcal{R}\|_{\mathcal{L}(\mathcal{H}^*(\p Y), \mathcal{H}(\p X))} &=& \delta^{1+m}\| \widetilde{\mathcal{R}}\|_{\mathcal{L}(\mathcal{H}^*(\p \widetilde{Y}), \mathcal{H}(\p \widetilde{X}))}.
\end{eqnarray*}

\end{lem}

\begin{proof}
The result follows from Lemma \ref{lem-scaling-1} and the following identity
$$
\mathcal{R} = \delta^{2+m} \eta^{-1} \circ \widetilde{\mathcal{R}} \circ \eta.
$$
\end{proof}

We first consider the operators $\mathcal{S}_{D_j}^{k}$ and $(\mathcal{K}_{D_j}^{k})^*$. The following asymptotic expansions hold.

\begin{lem}  \label{lem-appendix21}

\begin{enumerate}
\item[(i)]
Regarded as operators from $\mathcal{H}^*(\p D_j)$ into $\mathcal{H}(\p D_j)$, we have $$\mathcal{S}_{D_j}^{k} = \mathcal{S}_{D_j} + k \mathcal{S}_{D_j, 1} + k^2\mathcal{S}_{D_j, 2}  + O(k^3\delta^3),$$ where
$\mathcal{S}_{D_j} = O(1)$ and $\mathcal{S}_{D_j, m}= O(\delta^m)$;

\item[(ii)]
Regarded as operators from $\mathcal{H}(\p D_j)$ into $\mathcal{H}^*(\p D_j)$, we have
$$(\mathcal{S}_{D_j}^{k})^{-1} = \mathcal{S}_{D_j}^{-1} + k \mathcal{B}_{D_j, 1} + k^2\mathcal{B}_{D_j, 2} + O(k^3\delta^3),$$ where
$\mathcal{S}_{D_j}^{-1} = O(1)$ and $\mathcal{B}_{D_j, m} =O(\delta^m)$;

\item[(iii)]
Regarded as operators from $\mathcal{H}^*(\p D_j)$ into $\mathcal{H}^*(\p D_j)$, we have $$(\mathcal{K}_{D_j}^{k})^* = \mathcal{K}_{D_j}^* + k^2 O(\delta^2),$$ where $\mathcal{K}_{D_j}^*=O(1)$.

\end{enumerate}
\end{lem}

\begin{proof}
The proof immediately  follows from Lemmas \ref{lem-scaling}, \ref{lem-appendix11}, and \ref{lem-appendix13}.
\end{proof}

We now consider the operator $\mathcal{S}_{D_j, D_l}^{k}$.
By definition,
$$
\mathcal{S}_{D_j, D_l}^{k} [\psi] (x) = \int_{\p D_j} G(x, y, k) \psi(y) d\sigma(y),  \quad x \in  \p {D}_l.
$$
Using the expansion
$$
G(x, y, k)= \sum_{m=0}^{\infty} k^m Q_m(x, y),
$$
where $$ Q_m(x, y)= - \f{i^m |x-y|^{m-1}}{4 \pi},$$
we can derive that
$$
\mathcal{S}_{D_j, D_l}^{k} = \sum_{m \geq 0}k^m \mathcal{S}_{j, l, m},
$$
where
$$
\mathcal{S}_{j, l, m} [\psi] (x) = \int_{\p D_j} Q_m(x,y) \psi(y)d\sigma(y).
$$
We can further write  $$ \mathcal{S}_{j, l, m}= \sum_{n \geq 0} \mathcal{S}_{j, l, m, n},$$ where
$\mathcal{S}_{j, l, m, n}$ is defined by
$$
\mathcal{S}_{j, l, m, n} [\psi] (x) = 
\int_{\p D_j} \sum_{|\alpha| +|\beta|=n} \frac{1}{\alpha! \beta!} \f{\p^{|\alpha|+ |\beta|}}{\p x^{\alpha} \p y^{\beta}} Q_m(z_l, z_j) (x-z_l)^{\alpha}(y-z_j)^{\beta} \psi(y)\, d\sigma(y).
$$
In particular, we have
\beas
\mathcal{S}_{j, l, 0, 0}[\psi](x) &=&
- \f{1}{4\pi|z_j-z_l|} ( \psi, \chi({\p D_j}))_{H^{-1/2}(\p D_j),H^{1/2}(\p D_j) }\chi(D_l) , \\
\mathcal{S}_{j, l, 0, 1}[\psi](x)
 &=& \sum_{|\alpha|=1} \f{(z_l-z_j)^{\alpha}}{4\pi|z_l-z_j|^3} \Big( (x-z_l)^{\alpha} (\psi, \chi({\p D_l}))_{H^{-1/2}(\p D_j),H^{1/2}(\p D_j)} + \big( (y-z_j)^{\alpha}, \psi \big)\chi(D_l) \Big), \\
\mathcal{S}_{j, l, 0, 2}[\psi](x)
 &=&  \sum_{|\alpha|+ |\beta|=2} \frac{1}{\alpha! \beta!} \f{\p^{2} Q_0(z_l, z_j)}{\p x^{\alpha} \p y^{\beta}}(x-z_l)^{\alpha}(y-z_j)^{\beta} \psi(y)d\sigma(y), \\
\mathcal{S}_{j, l, 1}[\psi](x) &=&
- \f{i}{4\pi} (\psi, \chi({\p D_j}))_{H^{-1/2}(\p D_j),H^{1/2}(\p D_j)}\chi(D_l), \\
\mathcal{S}_{j, l, 2, 0}[\psi](x) &=&
\f{1}{4 \pi} |z_l-z_j| (\psi, \chi({\p D_j}))_{H^{-1/2}(\p D_j),H^{1/2}(\p D_j)}\chi(D_l).
\eeas

The following estimate holds.
\begin{lem} We have
$ \| \mathcal{S}_{j, l, m, n} \|_{\mathcal{L}(\mathcal{H}^*(\p D), \mathcal{H}(\p D))} \lesssim O(\delta^{n+1})$.
\end{lem}

\begin{proof}
After a translation of coordinates, the stated estimate immediately follows from Lemma \ref{lem-scaling}.
\end{proof}

Similarly, for the operator $\mathcal{K}_{D_j, D_l}^{k_m}$ defined in the following way
$$
\mathcal{K}_{D_j, D_l}^{k} [\psi] (x) = \int_{\p D_j} \f{ \p G(x, y, k)}{\p \nu(x)} \psi(y) d\sigma(y),  \quad x \in  \p {D}_l,
$$
we have
$$
\mathcal{K}_{D_j, D_l}^{k} = \sum_{m \geq 0}k^m \sum_{n \geq 0} \mathcal{K}_{j, l, m, n},
$$
where
$$
\mathcal{K}_{j, l, m, n} [\psi] (x) = \int_{\p D_j} \sum_{|\alpha|+|\beta|=n} \frac{1}{\alpha! \beta!} \f{\p^{n} K_m(z_l, z_j)}{\p x^{\beta} \p y^{\alpha}}  (x-z_l)^{\beta}(y-z_j)^{\alpha} (x-y)\cdot \nu(x) \psi(y) d\sigma(y)
$$
with $$K_m ( x, y) = - \f{i^m (m-1)|x-y|^{m-3}}{4 \pi m!}.$$
In particular, we have
\bea
\mathcal{K}_{j, l, 0, 0}[\psi](x) &=&
\f{1}{4\pi |z_l- z_j|^3}
\Big[  (x-z_l) \cdot \nu(x) \big(\psi, \chi(\p D_j)\big)_{H^{-1/2}(\p D_j),H^{1/2}(\p D_j)} \nonumber \\
&& - \big ( \psi, (y-z_j) \cdot \nu(x)\big)_{H^{-1/2}(\p D_j),H^{1/2}(\p D_j)} \nonumber \\
&& + (z_l-z_j) \cdot \nu(x)\big ( \psi , \chi(\p D_j)\big)_{H^{-1/2}(\p D_j),H^{1/2}(\p D_j)} \Big], \\
\mathcal{K}_{j, l, 1, m}[\psi] &=& 0 \quad \mbox{for all $m$}.
\eea

\begin{lem} We have
$ \| \mathcal{K}_{j, l, m, n} \|_{\mathcal{L}( \mathcal{H}^*(\p D_j), \mathcal{H}^*(\p D_l))} \lesssim O(\delta^{n+2})$.
\end{lem}

\begin{proof}
Note that
\beas
\mathcal{K}_{j, l, m, n} [\psi] (x) &=& \int_{\p D_j} \sum_{|\alpha|+|\beta|=n} \frac{1}{\alpha! \beta!} \f{\p^{n} K_m(z_l, z_j)}{\p x^{\beta} \p y^{\alpha}}  (x-z_l)^{\beta}(y-z_j)^{\alpha} (x-z_l)\cdot \nu(x) \psi(y) d\sigma(y) , \\
&& - \int_{\p D_j} \sum_{|\alpha|+|\beta|=n} \frac{1}{\alpha! \beta!} \f{\p^{n} K_m(z_l, z_j)}{\p x^{\beta} \p y^{\alpha}}  (x-z_l)^{\beta}(y-z_j)^{\alpha} (y-z_j)\cdot \nu(x) \psi(y) d\sigma(y) , \\
&& + \int_{\p D_j} \sum_{|\alpha|+|\beta|=n} \frac{1}{\alpha! \beta!} \f{\p^{n} K_m(z_l, z_j)}{\p x^{\beta} \p y^{\alpha}}  (x-z_l)^{\beta}(y-z_j)^{\alpha} (z_l-z_j)\cdot \nu(x) \psi(y) d\sigma(y).
\eeas
After a translation of coordinates, we can apply Lemma \ref{lem-scaling} to each one of the three terms above to conclude that
$\mathcal{K}_{j, l, m, n} = O(\delta^{n+3})+ O(\delta^{n+2})$. This completes the proof of the lemma.
\end{proof}

To summarize, we have proven the following results.
\begin{lem}  \label{lem-appendix22}
\begin{enumerate}
\item[(i)]
Regarded as an operator from $\mathcal{H}^{*}(\p D_j)$ into $\mathcal{H}(\p D_l)$ we have,
$$
\mathcal{S}_{D_j, D_l}^{k}= \mathcal{S}_{j,l,0,0} + \mathcal{S}_{j,l,0,1} + \mathcal{S}_{j,l,0,2}
+ k \mathcal{S}_{j,l, 1} + k^2 \mathcal{S}_{j,l,2,0} + O(\delta^4)+ O(k^2\delta^2).$$
 Moreover,$$
\mathcal{S}_{j,l,m,n}= O(\delta^{n+1}).
$$
\item[(ii)]
Regarded as an operator from $\mathcal{H}^{*}(\p D_j)$ into $\mathcal{H}^*(\p D_l)$, we have $$\mathcal{K}_{D_j, D_l}^{k}= \mathcal{K}_{j,l,0, 0} + O(k^2 \delta^2).$$ Moreover, $$\mathcal{K}_{j,l,0, 0}= O(\delta^2).$$
\end{enumerate}
\end{lem}

\section{Adaptation of results to the two-dimensional case} \label{appen2d}
In this section we adapt the layer potential formulation to plasmonic resonances in two dimensions. We only consider the single particle case. For the multiple particle case, a similar analysis holds. 

Recall that in $\mathbb{R}^2$ the single-layer potential $\mathcal{S}_D:H^{-1/2}(\partial D)\rightarrow H^{1/2}(\partial D)$ is not, in general, invertible nor injective. Hence, $-( u, \mathcal{S}_D[v])_{-\f{1}{2},\f{1}{2}}$ does not define an inner product and the symmetrization technique described in Lemma \ref{lem-Kstar_properties} is no longer valid.\\
To overcome this difficulty, a substitute of $\mathcal{S}_D$ can be introduced as in \cite{kang1} by
\be \label{eq-S_tilde}
\widetilde{\mathcal{S}}_D[\psi] = \left\{ \begin{array}{cc}
\mathcal{S}_D[\psi] & \quad \mbox{if } 
(\psi,\chi(\p D))_{-\f{1}{2},\f{1}{2}}=0, \\
1 & \quad \mbox{if } \psi=\varphi_0,
\end{array}
\right.
\ee
where $\varphi_0$ is the unique (in the case of a single particle) eigenfunction of $\mathcal{K}^*_D$ associated with eigenvalue $1/2$ such that $(\varphi_0,\chi(\p D))_{-\f{1}{2},\f{1}{2}}=1$. Note that, from the jump relations of the layer potentials, $\mathcal{S}_D[\varphi_0]$ is constant.\\
The operator $\widetilde{\mathcal{S}}_D:H^{-1/2}(\partial D)\rightarrow H^{1/2}(\partial D)$ is invertible. Moreover, the following Calder\'on identity holds $\mathcal{K}_D\widetilde{\mathcal{S}}_D = \widetilde{\mathcal{S}}_D \mathcal{K}^*_D $. With this, define
$$
(u,v)_{\mathcal{H}^*} = -( u, \widetilde{\mathcal{S}}_D[v])_{-\f{1}{2},\f{1}{2}}.
$$
Thanks to the invertibility and positivity of $-\widetilde{\mathcal{S}}_D$, this defines an inner product for which $\mathcal{K}^*_D$ is self-adjoint and $\mathcal{H}^*$ is equivalent to $H^{-1/2}$. Then, if $D$ is $\mathcal{C}^{1,\alpha}$, we have the following result. 
\begin{lem} \label{lem-K_star_properties2d}
Let $D$ be a $\mathcal{C}^{1,\alpha}$ bounded simply connected domain of $\mathbb{R}^2$ and let $\widetilde{\mathcal{S}}_D$ be the operator defined in \ref{eq-S_tilde}. Then,
\begin{enumerate}
\item[(i)]
The operator $\mathcal{K}_D^*$ is compact self-adjoint in the Hilbert space $\mathcal{H}^*(\p D)$ equipped with the 
inner product defined by
\be \label{innerproduct2}
(u, v)_{\mathcal{H}^*}= - (u, \widetilde{\mathcal{S}}_{D}[v])_{-\f{1}{2},\f{1}{2}}
\ee
with $(\cdot, \cdot)_{-\f{1}{2},\f{1}{2}}$ being the duality pairing between $H^{-1/2}(\p D)$ and  $H^{1/2}(\p D)$, which is equivalent to the original one;
\item[(ii)]
Let $(\lambda_j,\varphi_j) $, $j=0, 1, 2, \ldots,$ be the eigenvalue and normalized eigenfunction pair of $\mathcal{K}_D^*$ with $\lambda_0 = \f{1}{2}$.
Then, $\lambda_j \in (-\f{1}{2}, \f{1}{2}]$ and $\lambda_j \rightarrow 0$ as $j \rightarrow \infty$;
\item[(iii)]
$\mathcal{H}^*(\p D) = \mathcal{H}^*_0(\p D)\oplus\{\mu\varphi_0,\;\mu\in\mathbb{C}\}$, where $\mathcal{H}^*_0(\p D)$ is the zero mean subspace of $\mathcal{H}^*(\p D)$;
\item[(iv)]
The following representation formula holds: for any $\psi \in H^{-1/2}(\p D)$,
$$
\mathcal{K}_D^* [\psi]
= \sum_{j=0}^{\infty} \lambda_j (\psi, \varphi_j)_{\mathcal{H}^*} \otimes \varphi_j.
$$
\end{enumerate}
\end{lem}
\begin{rmk} \label{rmk-H*2d}
Note that $\widetilde{\mathcal{S}}_D^{-1}[\chi(\p D)] = \varphi_0$ and $(-\f{1}{2}Id+\mathcal{K}_D^*) = (-\f{1}{2}Id+\mathcal{K}_D^*)\mathcal{P}_{\mathcal{H}^*_0},$
where $\mathcal{P}_{\mathcal{H}^*_0}$ is the orthogonal projection onto ${\mathcal{H}^*_0}(\p D)$. In particular, we have $(-\f{1}{2}Id+\mathcal{K}_D^*)\widetilde{\mathcal{S}}_D^{-1}[\chi(\p D)]=0$.
\end{rmk}
Let us now consider the single-layer potential for the Helmholtz equation in $\mathbb{R}^2$
$$
\mathcal{S}_{D}^{k} [\psi](x) =  \int_{\p D} G(x, y, k) \psi(y) d\sigma(y),  \quad x \in  \p {D},
$$
where $G(x, y, k)= -\dfrac{i}{4}H_0^{(1)}(k|x-y|)$ and $H_0^{(1)}$ is the Hankel function of first kind and order $0$. We have
$$
-\dfrac{i}{4}H_0^{(1)}(k|x-y|) = \dfrac{1}{2\pi}\log |x-y|+\tau_k+\sum_{j=1}^{\infty}(b_j\log k|x-y|+c_j)(k|x-y|)^{2j},
$$
where
$$
\tau_k = \dfrac{1}{2\pi}(\log k+\gamma-\log 2)-\dfrac{i}{4}, \quad b_j = \dfrac{(-1)^j}{2\pi}\dfrac{1}{2^{2j}(j!)^2}, \quad c_j = -bj\left(\gamma -\log 2-\dfrac{i\pi}{2}-\sum_{n=1}^j\dfrac{1}{n}\right),
$$
and
$\gamma$ is the Euler constant. Thus, we get
\be \label{series-s2d}
\mathcal{S}_{D}^{k}=  \hat{\mathcal{S}}_{D}^k +\sum_{j=1}^{\infty} \left(k^{2j}\log k\right) \mathcal{S}_{D, j}^{(1)}+\sum_{j=1}^{\infty} k^{2j} \mathcal{S}_{D, j}^{(2)},
\ee
where
\beas
\hat{\mathcal{S}}_{D}^k[\psi](x) &=& \mathcal{S}_{D}[\psi](x) + \tau_k\int_{\partial D}[\psi]d\sigma,\\
\mathcal{S}_{D, j}^{(1)} [\psi](x) &=& \int_{\p D} b_j|x-y|^{2j} \psi(y)d\sigma(y),\\
\mathcal{S}_{D, j}^{(2)} [\psi](x) &=& \int_{\p D} |x-y|^{2j}(b_j\log|x-y|+c_j)\psi(y)d\sigma(y).
\eeas
\begin{lem} \label{lem-appendix11_2d}
The norms $\| \mathcal{S}_{D, j}^{(1)} \|_{\mathcal{L}({\mathcal{H}^*(\p D)}, \mathcal{H}(\p D))}$ and $\| \mathcal{S}_{D, j}^{(2)} \|_{\mathcal{L}({\mathcal{H}^*(\p D)}, \mathcal{H}(\p D))}$ are uniformly bounded with respect to $j$. Moreover, the series in (\ref{series-s2d}) is convergent in $\mathcal{L}({\mathcal{H}^*(\p D)}, \mathcal{H}(\p D))$.
\end{lem}
\begin{proof}
The proof is similar to that of Lemma \ref{lem-appendix11}.
\end{proof}
Observe that
$$
\left(\mathcal{S}_{D}-\widetilde{\mathcal{S}}_{D}\right)[\psi] =  \left(\mathcal{S}_{D}-\widetilde{\mathcal{S}}_{D}\right)[\mathcal{P}_{\mathcal{H}^*_0}[\psi]+(\psi,\varphi_0)_{\mathcal{H}^*}\varphi_0] = (\psi,\varphi_0)_{\mathcal{H}^*}\left(\mathcal{S}_D[\varphi_0]-\chi(\p D) \right).
$$
\\
Then it follows that
$$
\hat{\mathcal{S}}_{D}^k[\psi] = \widetilde{\mathcal{S}}_{D}[\psi]+(\psi,\varphi_0)_{\mathcal{H}^*}\left(\mathcal{S}_D[\varphi_0]-\chi(\p D) \right)+\tau_k\int_{\partial D}\psi_0+(\psi,\varphi_0)_{\mathcal{H}^*}\varphi_0 d\sigma = \widetilde{\mathcal{S}}_{D}[\psi]+\Upsilon_k[\psi],
$$
where \be \label{defUpsilon} \Upsilon_k[\psi] = (\psi,\varphi_0)_{\mathcal{H}^*}\left(\mathcal{S}_D[\varphi_0]-\chi(\p D) +\tau_k \right).\ee

Therefore, we arrive at the following result.
\begin{lem}
For $k$ small enough $\hat{\mathcal{S}}_{D}^{k} : \mathcal{H}^*(\p D) \rightarrow \mathcal{H}(\p D)$ is invertible.
\end{lem}
\begin{proof}
$\Upsilon_k$ is clearly a compact operator. Since $\widetilde{\mathcal{S}}_{D}$ is invertible, the invertibility of $\hat{\mathcal{S}}_{D}^{k}$ is equivalent to that of $\hat{\mathcal{S}}_{D}^{k}\widetilde{\mathcal{S}}_{D}^{-1} = Id + \Upsilon_k\widetilde{\mathcal{S}}_{D}^{-1}$. By the Fredholm alternative we only need to prove the injectivity of $Id + \Upsilon_k \widetilde{\mathcal{S}}_{D}^{-1}$.\\
Since $\forall \; v\in H^{1/2},\;\Upsilon_k \widetilde{\mathcal{S}}_{D}^{-1}[v]\in \mathbb{C}$, for $\left(Id + \Upsilon_k \widetilde{\mathcal{S}}_{D}^{-1}\right)[v]=0$, we need $v = \widetilde{\mathcal{S}}_{D}[\alpha \varphi_0] = \alpha \in \mathbb{C}$.\\
We have
$$
\left(Id + \Upsilon_k \widetilde{\mathcal{S}}_{D}^{-1}\right)\left[\widetilde{\mathcal{S}}_{D}[\alpha \varphi_0]\right] = \alpha(\mathcal{S}_D[\varphi_0] +\tau_k ) = 0 \quad \mbox{iff} \quad \mathcal{S}_D[\varphi_0] = -\tau_k \mbox{ or } \alpha = 0.
$$
Since we can always find a small enough $k$ such that $\mathcal{S}_D[\varphi_0] \neq -\tau_k$, we need $\alpha = 0$. This yields the stated result.
\end{proof}

\begin{lem}
For $k$ small enough, the operator $\mathcal{S}_{D}^{k} : \mathcal{H}^*(\p D) \rightarrow \mathcal{H}(\p D) $ is invertible.
\end{lem}
\begin{proof} The operator
$\mathcal{S}_{D}^{k}-\hat{\mathcal{S}}_{D}^{k}$ is a compact operator. Because $\hat{\mathcal{S}}_{D}^{k}$ is invertible for $k$ small enough, by the Fredholm alternative only the injectivity of $\mathcal{S}_{D}^{k}$ is necessary. From the uniqueness of a solution to the Helmholtz equation we get the result.\\
\end{proof}
We can write \eqref{series-s2d} as
$$
\mathcal{S}_{D}^{k}=  \hat{\mathcal{S}}_{D}^k+\mathcal{G}_k,
$$
where $\mathcal{G}_k = k^2\log k \mathcal{S}_{D,1}^{(1)} + k^2 \mathcal{S}_{D,1}^{(2)} + O(k^4\log k)$. From the two lemmas above we get the identity
$$
(\mathcal{S}_{D}^{k})^{-1} = \left(Id + (\hat{\mathcal{S}}_{D}^k)^{-1}\mathcal{G}_k\right)^{-1}(\hat{\mathcal{S}}_{D}^k)^{-1}.
$$
It is clear that $\|(\hat{\mathcal{S}}_{D}^k)^{-1}\|_{\mathcal{L}\left(\mathcal{H}(\partial D),\mathcal{H}^*(\partial D)\right)}$ is bounded in $k$. Thus, for $k$ small enough, we can formally write
$$
(\mathcal{S}_{D}^{k})^{-1} = (\hat{\mathcal{S}}_{D}^k)^{-1}-(\hat{\mathcal{S}}_{D}^k)^{-1}\mathcal{G}_k(\hat{\mathcal{S}}_{D}^k)^{-1}  +O(k^4\log^2 k).
$$
We have the identity $$ (\hat{\mathcal{S}}_{D}^k)^{-1} = \underbrace{\left( \widetilde{\mathcal{S}}_{D}^{-1} \hat{\mathcal{S}}_{D}^k\right)^{-1}}_{\Lambda_k^{-1}}\widetilde{\mathcal{S}}_{D}^{-1}.$$
Here,
$$
\Lambda_k = Id + (\cdot,\varphi_0)_{\mathcal{H}^*}(\mathcal{S}_D[\varphi_0]-\chi(\p D)+\tau_k)\varphi_0.
$$
Then,
$$
\Lambda_k^{-1} = Id - (\cdot,\varphi_0)_{\mathcal{H}^*}\frac{\mathcal{S}_D[\varphi_0]-\chi(\p D)+\tau_k}{\mathcal{S}_D[\varphi_0]+\tau_k}\varphi_0,
$$
and therefore,
$$
(\hat{\mathcal{S}}_{D}^k)^{-1} =  \widetilde{\mathcal{S}}_{D}^{-1} - (\widetilde{\mathcal{S}}_{D}^{-1} [\cdot],\varphi_0)_{\mathcal{H}^*}\varphi_0+\dfrac{(\widetilde{\mathcal{S}}_{D}^{-1} [\cdot],\varphi_0)_{\mathcal{H}^*}}{\mathcal{S}_D[\varphi_0]+\tau_k}\varphi_0.
$$
Finally, we get 
$$
\begin{array}{lll}
(\mathcal{S}_{D}^{k})^{-1} &=&\ds \mathcal{L}_D+\mathcal{U}_k-k^2\log k \mathcal{L}_D\mathcal{S}_{D,1}^{(1)}\mathcal{L}_D -k^2\left(\mathcal{L}_D\mathcal{S}_{D,1}^{(2)}\mathcal{L}_D - \log k (\mathcal{U}_k \mathcal{S}_{D,1}^{(1)}\mathcal{L}_D +  \mathcal{L}_D\mathcal{S}_{D,1}^{(1)}\mathcal{U}_k )\right)\\
\nm
&& + O(k^2\log^{-1} k)
\end{array}$$
with $\mathcal{L}_D=\mathcal{P}_{\mathcal{H}^*_0}\widetilde{\mathcal{S}}_{D}^{-1}$ and $ \mathcal{U}_k= \dfrac{(\widetilde{\mathcal{S}}_{D}^{-1} [\cdot],\varphi_0)_{\mathcal{H}^*}}{\mathcal{S}_D[\varphi_0]+\tau_k}\varphi_0$. We note that $\mathcal{U}_k = O(\log^{-1} k)$.\\

We now consider the expansion for the boundary integral operator $(\mathcal{K}_{D}^{k})^*$. We have
\be \label{series-k2d}
(\mathcal{K}_{D}^{k})^* = \mathcal{K}_D^* +\sum_{j=1}^{\infty} \left(k^{2j}\log k\right) \mathcal{K}_{D, j}^{(1)}+\sum_{j=1}^{\infty} k^{2j} \mathcal{K}_{D, j}^{(2)},
\ee
where
\beas
\mathcal{K}_{D, j}^{(1)} [\psi](x) &=& \int_{\p D} b_j\dfrac{\partial |x-y|^{2j}}{\partial \nu(x)}\psi(y)d\sigma(y),\\
\mathcal{K}_{D, j}^{(2)} [\psi](x) &=& \int_{\p D} \dfrac{\partial \left( |x-y|^{2j}(b_j\log|x-y|+c_j)\right)}{\nu(x)}\psi(y)d\sigma(y).
\eeas

\begin{lem} \label{lem-appendix132d} The norms
$\| \mathcal{K}_{D, j}^{(1)} \|_{ \mathcal{L}(\mathcal{H}^*(\p D), \mathcal{H}^*(\p D))}$ and $\| \mathcal{K}_{D, j}^{(2)} \|_{\mathcal{L}( \mathcal{H}^*(\p D), \mathcal{H}^*(\p D))}$ are uniformly bounded for $j \geq 1$.
Moreover, the series in (\ref{series-k2d}) is convergent in $\mathcal{L}({\mathcal{H}^*(\p D)}, {\mathcal{H}^*(\p D)})$.
\end{lem}
\begin{proof}
The proof is similar to that of Lemma \ref{lem-appendix11}.
\end{proof}

Recall \eqref{Aw} and \eqref{deff}, we can show that the following result holds.
\begin{lem}
Regarding ${\mathcal{A}_{D}}(\om)$ as an operator from $\mathcal{H}^*(\p D)$ to $\mathcal{H}^*(\p D)$, we have
$$
{\mathcal{A}_{D}}(\om)= {\mathcal{A}}_{D, 0} + \om^2 (\log \om) {\mathcal{A}}_{D, 1} + O(\om^2),
$$
where
\beas
\mathcal{A}_{D,0} &=& \big( \f{1}{2\mu_m} +  \f{1}{2\mu_c}\big)Id + \big( \f{1}{\mu_m} -  \f{1}{\mu_c} \big) \mathcal{K}_D^* ,\\
\mathcal{A}_{D, 1} &=& \mathcal{K}_{D, 1}^{(1)}\left( \eps_mId-\eps_c\mathcal{P}_{\mathcal{H}^*_0} \right) +
\f{1}{\mu_c} (\f{1}{2}Id - \mathcal{K}_D^*)
\widetilde{\mathcal{S}}_D^{-1}\mathcal{S}_{D,1}^{(1)}\left( \mu_m\eps_mId-\mu_c\eps_c\mathcal{P}_{\mathcal{H}^*_0} \right).
\eeas
\end{lem}
\begin{proof}
We have
\beas
(\mathcal{S}_D^{k_c})^{-1} &=&  \mathcal{L}_D+\mathcal{U}_{k_c}-\om^2 (\log \om) \eps_c\mu_c\mathcal{L}_D\mathcal{S}_{D,1}^{(1)}\mathcal{L}_D + O(\om^2)\\
\mathcal{S}_D^{k_m} &=& \widetilde{\mathcal{S}}_{D}+\Upsilon_{k_m} + \om^2 (\log\om) \eps_m\mu_m\mathcal{S}_{D,1}^{(1)} + O(\om^2).
\eeas
Also, $\mathcal{L}_D \Upsilon_{k_m} = \mathcal{P}_{\mathcal{H}^*_0}(\widetilde{\mathcal{S}}_{D})^{-1} \Upsilon_{k_m} = 0$, where $\Upsilon_{k_m}$ is defined by (\ref{defUpsilon}). Hence,
\beas
(\mathcal{S}_D^{k_c})^{-1} \mathcal{S}_D^{k_m} &=& \mathcal{P}_{\mathcal{H}^*_0} + \mathcal{U}_{k_c}\widetilde{\mathcal{S}}_{D} + \mathcal{U}_{k_c}\Upsilon_{k_m} +  \om^2(\log\om) \big(\eps_m\mu_m\mathcal{L}_D\mathcal{S}_{D,1}^{(1)} -\eps_c\mu_c\mathcal{L}_D\mathcal{S}_{D,1}^{(1)}\mathcal{L}_D\widetilde{\mathcal{S}}_{D}\big) + O(\om^2)\\
&=& \mathcal{P}_{\mathcal{H}^*_0} + \mathcal{U}_{k_c}\widetilde{\mathcal{S}}_{D} + \mathcal{U}_{k_c}\Upsilon_{k_m} +  \om^2\log\om\mathcal{L}_D\mathcal{S}_{D,1}^{(1)}\big(\eps_m\mu_m Id -\eps_c\mu_c\mathcal{P}_{\mathcal{H}^*_0}\big) + O(\om^2).
\eeas
From Remark \ref{rmk-H*2d}, it follows that
$$
\big(\f{1}{2}Id - \mathcal{K}_D^*\big)\mathcal{U}_{k_c}=0.
$$
Since $\f{1}{2}Id - (\mathcal{K}_D^{k})^* = \big( \f{1}{2}Id - \mathcal{K}_D^* \big) - k^2\log k \mathcal{K}_{D, 1}^{(1)} + O(k^2) $, we get the desired result.
\end{proof}

Under Conditions \ref{condition1} and \ref{condition1add}, the perturbed eigenvalues and eigenfunctions of ${\mathcal{A}_{D}}(\om)$ have the following form
\bea
\tau_j (\om) &=& \tau_j + \om^2 (\log \om) \tau_{j, 1} + O(\om^2), \label{tau-single2d} \\
\varphi_j(\om) &=& \varphi_j + \om^2 (\log \om) \varphi_{j, 1} + O(\om^2), \label{eigenfun-single2d}
\eea
where
\bea
\tau_{j, 1} &=& R_{j j}, \\
\varphi_{j, 1}&=& \sum_{l\neq j} \f{R_{jl}}{ \big( \f{1}{\mu_m} -  \f{1}{\mu_c} \big) (\lambda_j- \lambda_l)} \varphi_l,
\eea
and
$$
R_{jl} = (\mathcal{A}_{D,1}[\varphi_j],\varphi_l)_{\mathcal{H}^*}.
$$
It is clear that Lemma \ref{lem-residu} holds in the two-dimensional case. We also have the following asymptotic expansion for $f$ in terms of $\omega$.

\begin{lem} \label{lem-f2d}
In the space $\mathcal{H}^*(\partial D)$,  as $\omega$ goes to zero, we have
$$
f = \om f_1 + O(\om^2),
$$
where $$f_1= -ie^{ik_md\cdot z} \sqrt{\eps_m\mu_m}\left(\f{1}{\mu_m}[d\cdot\nu(x)]+\f{1}{\mu_c}(\dfrac{1}{2}Id - \mathcal{K}_{D}^*) \widetilde{\mathcal{S}}_{D}^{-1}[d\cdot (x-z)]\right)$$ and $z$ is the center of the domain $D$.
\end{lem}

Finally, the following result holds.
\begin{thm} \label{thm12d}
Under Conditions \ref{condition0}, \ref{condition1}, and \ref{condition1add},
the scattered field by a single plasmonic particle, $u^s = u - u^i$,  has in the quasi-static limit the following representation:
$$
u^s = \mathcal{S}_D^{k_m} [\psi],
$$
where
\beas
\psi = \sum_{j \in J}\f{i k_m e^{ik_md\cdot z}\big(d\cdot\nu(x) ,\varphi_j\big)_{\mathcal{H}^*} \varphi_j +O(\om^3\log\om)}{ \lambda - \lambda_j +O(\om^2\log\om)}
+  O(\om)
\eeas
with $\lambda$ being defined by (\ref{deflambda}).
\end{thm}
\begin{proof}
We have
\beas
\psi &=& \sum_{j \in J} \f{\big( f, \widetilde{\varphi}_j(\om)\big)_{\mathcal{H}^*} \varphi_j(\om) }{ \tau_j(\om)}
+  {\mathcal{A}_{D}}(\om)^{-1}  ( P_{J^c}(\om) f) \\
 &=& \sum_{j \in J}\f{\om\big( f_1, \varphi_j\big)_{\mathcal{H}^*} \varphi_j + O(\om^3(\log\om))}{ \f{1}{2\mu_m} +  \f{1}{2\mu_c} - \big( \f{1}{\mu_c}-\f{1}{\mu_m}  \big) \lambda_j + O(\om^2\log\om)}
+  O(\om).
\eeas
Since $d\cdot (x-z)$ is a harmonic function, changing $\mathcal{S}_D$ by $\widetilde{\mathcal{S}}_D$ in Theorem \ref{thm1} yields
$$
\Big( (\dfrac{1}{2}Id - \mathcal{K}_{D}^*) \mathcal{S}_{D}^{-1}[d\cdot (x-z)], \varphi_j\Big)_{\mathcal{H}^*} = -( d\cdot\nu(x) ,\varphi_j)_{\mathcal{H}^*}.
$$
Then, the proof is complete. 
\end{proof}

\begin{cor}
Assume the same conditions as in Theorem \ref{thm1}. Then, under the additional condition 
$$
\min_{j\in J} |\tau_j(\om)| \gg \om^2,
$$
we have
$$
\psi = \sum_{j \in J} \f{i k_m e^{ik_md\cdot z}\big(d\cdot\nu(x) ,\varphi_j\big)_{\mathcal{H}^*} \varphi_j +O(\om^3\log\om)}{ \lambda - \lambda_j + \om^2\log\om \big( \f{1}{\mu_c}-\f{1}{\mu_m}  \big)^{-1}\tau_{j, 1}}
+  O(\om).
$$
\end{cor}



\section{Sum rules for the polarization tensor} \label{appendixPT}
Let $f$ be a holomorphic function defined in an open set $U\subset \mathbb{C}$ containing the spectrum of $\mathcal{K}^*_{\partial D}$. Then, we can write $f(z) = \displaystyle\sum_{j=0}^{\infty}a_jz^j$ for every $z\in U$.\\
\begin{definition} Let
$$
f(\mathcal{K}^*_{D}):=\displaystyle\sum_{j=0}^{\infty} a_j(\mathcal{K}^*_{D})^j,$$
where $(\mathcal{K}^*_{D})^j:= \underbrace{\mathcal{K}^*_{ D}\circ\mathcal{K}^*_{D}\circ..\circ \mathcal{K}^*_{D}}_{j\;times}.$
\end{definition}

\begin{lem} We have
$$
f(\mathcal{K}^*_{D})=\displaystyle\sum_{j=1}^{\infty} f(\lambda_j) ( \cdot , \varphi_j )_{\mathcal{H}^*} \varphi_j.
$$
\label{lemma:HolomFunc}
\end{lem}
\begin{proof}
We have
$$
\begin{array}{lll}
f(\mathcal{K}^*_{D})&=& \ds\sum_{i=0}^{\infty} a_i(\mathcal{K}^*_{D})^i = \sum_{i=0}^{\infty} a_i\sum_{j=1}^{\infty} \lambda_j^i ( \cdot , \varphi_j )_{\mathcal{H}^*} \varphi_j \\
\nm
 &=& \ds \sum_{j=1}^{\infty} \left( \sum_{i=0}^{\infty} a_i \lambda_j^i \right) ( \cdot , \varphi_j )_{\mathcal{H}^*} \varphi_j \\
 \nm
 &= & \ds\sum_{j=1}^{\infty} f(\lambda_j) ( \cdot , \varphi_j )_{\mathcal{H}^*} \varphi_j.
\end{array}$$
\end{proof}
From Lemma \ref{lemma:HolomFunc}, we can deduce that
\begin{equation}
\int_{\partial D} x_l f(\mathcal{K}^*_{D})[\nu_m](x)\, d\sigma(x)  = \sum_{j=1}^{\infty}f(\lambda_j)\alpha^{(j)}_{l,m}.
\label{eq:SumRules_HolomFunc}
\end{equation}
Equation \eqref{eq:SumRules_HolomFunc} yields the summation rules for the entries of the polarization tensor.\\

In order to prove that $\displaystyle\sum_{j=1}^{\infty}\alpha^{(j)}_{l,m} = \delta_{l,m}|D|$,
we take $f(\lambda)=1$ in \eqref{eq:SumRules_HolomFunc} to get
\begin{equation*}
\sum_{j=1}^{\infty}\alpha^{(j)}_{l,m} = \int_{\partial D} x_l\nu_m(x)\,  d\sigma(x) =\delta_{l,m}|D|.
\end{equation*}

Next, we prove that $$\displaystyle\sum_{j=1}^{\infty}\lambda_j\displaystyle\sum_{l=1}^d\alpha^{(j)}_{l,l} = \dfrac{(d-2)}{2}|D|.$$
Taking $f(\lambda)=\lambda$ in \eqref{eq:SumRules_HolomFunc}, we obtain
\begin{align*}
\sum_{j=1}^{\infty}\lambda_j\sum_{l=1}^d\alpha^{(j)}_{l,l} &= \sum_{l=1}^d \int_{\partial D} x_l \mathcal{K}^*_{D}[\nu_l](x)\, d\sigma(x),\\
\int_{\partial D} x_l \mathcal{K}^*_D[\nu_l](x)\, d\sigma(x) =& \int_{\partial D} x_l \left( \frac{1}{2} \nu_l(x) +\frac{\partial \mathcal{S}_D[\nu_l](x)}{\partial \nu}  \Big\vert_- \right) d\sigma(x) ,\\
 =& \frac{\vert D \vert }{2} + \int_{\partial D} x_l \frac{\partial \mathcal{S}_D[\nu_l](x)}{\partial \nu}\Big\vert_-  d\sigma(x). \numberthis \label{eq:eq_i}
\end{align*}
Integrating by parts we arrive at
\begin{align*}
\int_{\partial D} x_l \frac{\partial \mathcal{S}_D[\nu_l](x)}{\partial \nu} \Big\vert_-(x)  d\sigma(x) = \int_D e_l(x) \cdot \nabla \mathcal{S}_D[\nu_l](x)dx + \int_{D} x_l \Delta \mathcal{S}_D[\nu_l](x) dx.
\end{align*}
Since the single-layer potential is harmonic on $D$,
\begin{align*}
\int_{\partial D} x_l \frac{\partial \mathcal{S}_D[\nu_l](x)}{\partial \nu}\Big\vert_-(x)  d\sigma(x) = \int_D e_l(x) \cdot \left(\int_{\partial D} \nabla_x \Gamma(x,x') \nu_l(x') d\sigma(x')\right)\, dx.
\end{align*}
Summing on $i$  and using $\nabla_x \Gamma(x,x')=-\nabla_{x'} \Gamma(x,x')$, we get
\begin{align*}
\sum_{l=1}^d\int_{\partial D} x_l \frac{\partial \mathcal{S}_D[\nu_l]}{\partial \nu(x)}\Big\vert_-(x)  d\sigma(x) =& -\int_D \left( \int_{\partial D} \nu(x')\cdot \nabla_{x'}\Gamma(x,x') d\sigma(x') \right)dx , \\
=& -\int_D \mathcal{D}_D[1](x) dx ,\\
=& -|D|  ,\numberthis \label{eq:eq_i+1}
\end{align*}
where $\mathcal{D}_D$ is the double-layer potential. 
Hence, summing equation (\ref{eq:eq_i}) for $i=1,\ldots, d,$  we get the result.

Finally, we show that $$\displaystyle\sum_{j=1}^{\infty}\lambda_j^2\displaystyle\sum_{l=1}^d\alpha^{(j)}_{l,l} = \dfrac{d-4}{4}|D|+\sum_{l=1}^d\int_{ D} |\nabla \mathcal{S}_{D}[\nu_l]|^2dx.$$
Taking $f(\lambda)=\lambda^2$ in \eqref{eq:SumRules_HolomFunc} yields
\begin{align*}
\sum_{j=1}^{\infty}\lambda_j^2\sum_{l=1}^d\alpha^{(j)}_{l,l} &= \sum_{l=1}^d \int_{\partial D} x_l(\mathcal{K}^*_{D})^2[\nu_l](x)\, d\sigma(x)\\
& = \sum_{l=1}^d\int_{\partial D} \mathcal{K}_{D}[y_l](x) \mathcal{K}^*_{D}[\nu_l](x)\, d\sigma(x) \\
& = \sum_{l=1}^d\int_{\partial D} \mathcal{K}_{D}[y_l]\dfrac{\nu_l}{2}d\sigma + \sum_{l=1}^d\int_{\partial D} \mathcal{K}_{D}[y_l]\dfrac{\partial \mathcal{S}_{D}[\nu_l]}{\partial \nu}|_{-}d\sigma \\
& = \dfrac{(d-2)}{4}|D|-\underbrace{\sum_{l=1}^d\int_{\partial D}\dfrac{y_l}{2}\dfrac{\partial \mathcal{S}_{D}[\nu_l]}{\partial \nu}\Big\vert_-d\sigma}_{I_1} + \underbrace{\sum_{l=1}^d\int_{\partial D}\mathcal{D}_{D}[y_l]\Big\vert_-\dfrac{\partial \mathcal{S}_{D}[\nu_l]}{\partial \nu}\Big\vert_-d\sigma}_{I_2}.
\end{align*}
From \eqref{eq:eq_i+1} it follows that
\begin{align*}
I_1 = -\dfrac{|D|}{2}.
\end{align*}
Since $x_l$ is harmonic, we have $x_l = \mathcal{D}_{D}[y_l](x)|_{-}-\mathcal{S}_{D}[\nu_l](x)$ on $\partial D$, and thus,
\begin{align*}
I_2 &= \sum_{l=1}^d\int_{\partial D}\left(x_l+\mathcal{S}_{D}[\nu_l](x)\right)\dfrac{\partial \mathcal{S}_{D}[\nu_l](x)}{\partial \nu}\Big\vert_-d\sigma(x) ,\\
& = -|D|+\sum_{l=1}^d\int_{\partial D}\mathcal{S}_{D}[\nu_l]\dfrac{\partial \mathcal{S}_{D}[\nu_l]}{\partial \nu}\Big\vert_-d\sigma,\\
& = -|D|+ \sum_{l=1}^d\int_{ D} |\nabla \mathcal{S}_{D}[\nu_l]|^2dx.\\
\end{align*}
Replacing $I_1$ and $I_2$ by their expressions gives the desired result.

%
%

\end{document}